\documentclass[3p]{elsarticle}
\usepackage{amsmath}
\usepackage{amsfonts}
\usepackage{amssymb}
\usepackage{amsthm}
\usepackage{xcolor}

\usepackage{tikz}
\usetikzlibrary{decorations.pathreplacing,positioning, arrows.meta}
\usepackage{multirow}
\usepackage{multicol}
\usepackage[utf8]{inputenc}
\usepackage{float}
\usepackage{graphicx}
\usepackage{pdfpages}

\usepackage{lineno,hyperref}
\usepackage{subfig}
\modulolinenumbers[5]

\newtheorem{lemma}{Lemma}

\newtheorem{definition}{Definition}
\newtheorem{example}{Example}
\usepackage[mode=buildnew]{standalone}
\usepackage{tikz}
\usepackage{appendix}
\usepackage{listings}
\usepackage{slashbox}
\usepackage{caption}
\usepackage{subfig}
\usepackage{hyperref}

\usepackage{colortbl}
\definecolor{gris0}{gray}{0.98}
\definecolor{gris1}{gray}{0.93}
\definecolor{gris2}{gray}{0.84}
\definecolor{gris3}{gray}{0.75}
\definecolor{gris4}{gray}{0.69}
\definecolor{gris5}{gray}{0.63}
\definecolor{gris6}{gray}{0.57}

\begin{document}

\begin{frontmatter}

\title{Condition-based maintenance for a system subject to multiple degradation processes with stochastic arrival intensity}


\author{Lucía Bautista} 
\address{Department of Mathematics \\ University of Extremadura, Cáceres (Spain) \\ email: luciabb@unex.es}

\author{Inma T. Castro*}
\address{Department of Mathematics \\ University of Extremadura, Cáceres (Spain) \\ email: inmatorres@unex.es}

\author{Luis Landesa} 
\address{Department of Computers and Communications Technology \\ University of Extremadura, Cáceres (Spain) \\ email: llandesa@unex.es}




\cortext[mycorrespondingauthor]{Corresponding author}


\begin{abstract}

In this work, a system subject to different deterioration processes is analysed. The arrival of the degradation processes to the system is modelled using a shot-noise
Cox process. The degradation processes grow according to an homogeneous gamma process. The system fails when a degradation process exceeds a failure threshold. The
combined process of initiation and growth of the degradation processes is modelled and the system reliability is obtained. Heterogeneities are {\color{black} also} integrated in the model assuming that the {\color{black} inverse of the} scale parameter follows a uniform distribution. 
{\color{black} A maintenance strategy is implemented in this system and the state of the system is checked in inspection times}. If the system is {\color{black} working at inspection time}, a preventive replacement is performed if the deterioration level of a degradation process exceeds a certain threshold. A corrective replacement is performed if the system {\color{black} is down at inspection time}. Under this maintenance strategy, the expected cost rate is obtained. {\color{black} Sensitivity} analysis on the main parameters of the gamma process is performed.

\end{abstract}
\begin{keyword}
condition-based maintenance, gamma process, heterogeneities, shot-noise Cox process, complex degradation processes.
\end{keyword}

\end{frontmatter}


\section{Introduction}
It is well known that Condition-Based Maintenance (CBM) is a maintenance strategy based on monitoring the operating condition of a system. Compared to time-based maintenance, and thanks to the development of sensor technologies, CBM usually results in lower maintenance costs, avoiding unnecessary preventive maintenance activities and reducing unexpected failures \cite{Huynh}. However, as Alaswad and Xiang \cite{Alaswad} claim, further research on CBM is still in great need. Since industrial systems are becoming
{\color{black} increasingly complex} and they are likely to suffer from multiple degradation processes, further research is needed in modelling the maintenance when a system possesses more than one degradation path. 
Systems subject to multiple degradation processes can be found in electronic products, heavy machine tools and piping systems \cite{Liu2020}. On a pavement network, several different degradation processes, such as fatigue cracking and pavement deformation, may develop simultaneously \cite{Wu}. 

In literature, significant and meaningful prior research has been done on reliability and maintenance policies for systems subject to multiple degradation processes. An earlier work on this topic 
is \cite{Crk}, where a multiple multivariate regression is applied to model the complexity of the processes. {\color{black} Some studies}  model the joint probability assuming that the degradation processes are independent (\cite{Castro}, \cite{Caballe}, \cite{Bordes}). It could be the case of components in different production units without interference. But, in some cases, interaction between adjacent processes has a significant influence on the propagation characteristics of the rest of the processes. In these cases, the dependence between degradation processes has to be taken into account. Different works that deal with dependent degradation processes can be found in \cite{Wang} and \cite{Castro2}.


In general, for complex degradation processes, there are also two approaches to model the degradation mechanism of the processes. The first approach considers that all the processes start to degrade at the same time (\cite{Jia}). However, as Kuniewski {\it et. al} claim \cite{Kuniewski}, it is unlikely that all processes appear at the same time. The second approach assumes that the processes initiate at random times and then grow depending on the environment and conditions of the system. In this approach, two stochastic processes have to be combined: the initiation process and the growth process. Let $\left\{N^*(t), t \geq 0\right\}$ be the point process of the degradation processes with arrival times $S_1<S_2< \ldots$ (``initiating events''). Each degradation process triggers the ``effective event'' of failure after a random time $D_i$, where $D_i$ are i.i.d. non-negative variables. The sequence $\left\{S_i+D_i\right\}$ forms a new point process. The initiating events can often be interpreted as potentially harmful events affecting a system.  The  ``wear process''
triggered is activated at the moment of the occurrence of the harmful event and continuously
increases with time \cite{Cha2012}.  This approach has been recently dealt by Cha and Filkenstein \cite{Cha2018b} assuming that the initiation process follows a generalized Polya process. They obtained the survival function of the system and an analysis of the residual lifetime was also performed.

This idea of combining the initiation process and the growth process is adopted in this paper. Some of the works that use this second approach assume that the degradation processes arrive either at a constant rate (using the Poisson process \cite{Huynh2017}) or governed by a deterministic intensity function (using the non-homogeneous Poisson process \cite{Caballe}). If the degradation processes appear following a non-homogeneous Poisson process (NHPP) and all the degradation processes degrade following the same degradation mechanism, the combined point process $\left\{S_i+D_i\right\}$ follows a non-homogeneous Poisson process. The non-homogeneous Poisson assumption facilitates explicit analysis and it has been applied to real data in order to model the corroding gas pipeline system over its lifetime \cite{Tee}. However, modelling the arrival process as a non-homogeneous Poisson process assumes that the initiation intensity is deterministic along the system lifetime.  {\color{black} The motivation of our work is the following: external shocks {\color{black} can accelerate the arrival of new degradation processes to the system hence the deterministic intensity of arrivals can not be settled. For example, in harsh environments, shocks can induce highly dynamic loads on structures causing cracking problems \cite{Huang}. In the case of cermets, many cracks are initiated due to mechanical shocks \cite{Ishiara}}}.  When shocks affect to the intensity of arrival of new degradation processes, this phenomenon is better captured by a shot noise Cox process instead of a non-homogeneous Poisson process. The {\color{black} shot noise} Cox process has been employed as a useful tool for modeling the impact on the system lifetime of a dynamic environment \cite{Cha2017}, \cite{Cha2018}, \cite{Cha2018c}. In this paper, we assume that the degradation processes initiate at random times under a shot noise Cox process. From its introduction in reliability by Lemoine and Wenocur \cite{Lemoine1} and \cite{Lemoine2}, different authors have used this process in reliability and maintenance mainly for modeling the increment of the failure rate or an abrupt degradation increment due to external shocks (see \cite{Finkelstein2020} and \cite{Qiu} as recent examples of the use of shot noise process in maintenance and reliability).

Sometimes, the degradation processes present substantial variations between them causing different degradation patterns. Different models have been proposed to integrate the heterogeneities under the common idea that some parameters are process-specific and different across processes \cite{ChenEJOR}. An approach to take into account the heterogeneities is to assume that the parameters of the model follow a random variable \cite{Chen}. In the case of a gamma process, the {\color{black} gamma distribution} in itself is a very attractive candidate for the distribution of the scale parameter \cite{Lawless}. Assuming the gamma distribution for the scale parameter, the joint distribution of the scale parameter and the gamma process has a closed form expression based on the Snedecor F distribution, which allows the computation of the failure time. This gamma random effect has been studied by different authors (\cite{Picon}, \cite{Xiang}, \cite{Tsai} among others). Pulcini \cite{Pulcini} extends the gamma random effect assuming that heterogeneity affects to the shape parameter. In this paper, unlike the articles mentioned above, we incorporate the random effects in the gamma process to model the process-specific heterogeneity assuming that the scale parameter follows a uniform distribution. As far as we are concerned, there are not many works that model the heterogeneity using a uniform distribution.

In short, in this paper, we deal with a system subject to multiple degradation processes. We assume that the degradation processes initiate according to a shot noise Cox process and grows according to a stationary gamma process. The system fails when a degradation process exceeds a failure threshold. An inspection policy is developed for this system and it is periodically inspected. {\color{black} At each inspection}, the decision on whether a preventive replacement or a corrective replacement should be taken, is performed. {\color{black} The analytical expression for the expected cost rate is obtained}.

The main contributions of this paper are summarized in the following aspects. 

\begin{itemize}
\item {\color{black} A shot noise Cox process is imposed as process of initiation of degradation processes. It extends and generalizes the combined model of initiation and growth that assumes a non homogeneous Poisson process of arrivals}. The methodology used in this paper allows to find easy-to-evaluate expressions for reliability expressions. 
\item {\color{black} A random effect model is incorporated to the model with the novelty of the probability distribution used to deal with the heterogeneity between processes. In this paper, a uniform distribution is used to model the inverse of the sc. It allows to evaluate the moments associated to this process and compare the variance between models with heterogeneity and without heterogeneity}. 
\item Providing the analytic cost model for the combined model of arrivals and growth. The proposed model allows to get the analytic expressions for the quantities related to the maintenance. 
\end{itemize}

This remainder of the paper is structured as follows. In Sections \ref{a} and \ref{b}, the arrival and growth processes are described. In Section \ref{c}, the combined process of initiation and growth is analyzed. Section \ref{c} explains also the heterogeneity model. In Section \ref{d}, the survival time is obtained and {\color{black} it is shown that the failure time distribution is increasing failure rate (IFR)}. Section \ref{e} is devoted to the maintenance analysis. Section \ref{f} shows numerical examples and last section conclude.

\section{Arrival processes} \label{a}

%
%



We assume that a system is working in a dynamic environment and it is subject to external shocks. The external shocks arrive to the system according to a Poisson process with deterministic rate $\mu$. Suppose that if a shock occurs at epoch $T_1$, then at time $T_1+{\color{black} s}$ the contribution of the shock to the arrival of a degradation process is $h({\color{black} s})=\exp(-\delta({\color{black} s}-T_1))$ with $\delta >0$. Let $T_i, \ i=1,2,\ldots$ be the arrival times of the homogeneous Poisson process with rate parameter $\mu$ and let
$N(s)= \sum_{i=1}^{\infty} {\bf 1}_{\lbrace T_i \leq {\color{black} s} \rbrace}$ be the counting process associated with the homogeneous Poisson process. 
Then, the failure rate of the shot noise Cox process {\color{black} at time $s$}, $\lambda^*({\color{black} s})$, is given by

\begin{equation} \label{snprocess}
\lambda^*({\color{black} s})=\lambda_0(s)+ \sum_{i=1}^{N({\color{black} s})} \exp{(-\delta ({\color{black} s}-T_i))}, \quad s \geq 0,
\end{equation}
{\color{black} where the deterministic function $\lambda_0(s) >0$ provides a Poisson base level for the process. In absence of external shocks or if $\delta \rightarrow \infty$, a shot noise Cox process reduces to a non homogeneous Poisson process. }

Since {\color{black} the intensity $\lambda^*$} given by Eq. (\ref{snprocess}) is stochastic, the expectation {\color{black} at time $s$} {\color{black} is given by}, 
\begin{eqnarray} \nonumber
\mathbb{E} (\lambda^*({\color{black} s})) &=& \lambda_0(s)+\mathbb{E}\left(\sum_{i=1}^{N({\color{black} s})} \exp{(-\delta ({\color{black} s}-T_i))}\right) \\ \label{expectationlambdastar}
&=& \lambda_0(s)+\frac{n}{{\color{black} s} \delta}\left(1-\exp(-\delta {\color{black} s})\right),
\end{eqnarray}
{\color{black} using} that, conditional on $\left\{N(s)=n\right\}$,  {\color{black} the vector} $(T_1, T_2, \ldots, T_n)$ has the same distribution as the order statistics of sample $(U_1, U_2, \ldots, U_n)$ of size $n$ from the distribution
$$P(U \leq {\color{black} t})=\frac{{\color{black} t}}{{\color{black} s}}, \quad 0 \leq {\color{black} t} \leq {\color{black} s},  $$
where $0<t_1<t_2<\ldots<t_n \leq {\color{black} s}$ (see \cite{Lemoine2} for more details). 

Next, we compute the expected number of degradation processes at time ${\color{black} s}$. For that, the following lemma, valid for all counting process whose proof is given in \cite{Ross} (pp. 335-336), is used. 

\begin{lemma}\label{lemaNumberArrivals} 
Let $\lambda^*({\color{black} s})$, with ${\color{black} s} \geq 0$, be the random intensity function of the counting process $\left\{N^*({\color{black} s}), {\color{black} s} \geq 0\right\}$ having $N^*(0)=0$. Then
\begin{eqnarray*}
\mathbb{E}[N^*({\color{black} s})]&=& \int_{0}^{{\color{black} s}} \mathbb{E}[\lambda^*(u)]~du. 
\end{eqnarray*}
\end{lemma}
Using Lemma \ref{lemaNumberArrivals}, the expected number of  degradation processes at time ${\color{black} s}$ is given by
\begin{eqnarray*}\nonumber
\mathbb{E}[N^*({\color{black} s})]  &=&  \int_0^{\color{black} s} \mathbb{E}\left[ \lambda^*(u)\right]~ du \\
&=& \Lambda_0(s)  + \frac{\mu {\color{black} s}}{\delta}+ \frac{\mu}{\delta^2} \left(\exp{(-\delta {\color{black} s})}-1\right).
\end{eqnarray*}
where
$$\Lambda_0(s)=\int_{0}^{s}\lambda_0(u)du. $$

\section{The stochastic process of growth} \label{b}
Once a degradation process arrives to the system, the process of growth is activated. We assume that degradation processes grow independently each other.  Due to its mathematical properties, a gamma process is used as mathematical model of the growth. The gamma process is a stochastic process with independent gamma-distributed increments. The gamma process with shape function $\alpha(t) > 0$ and scale parameter $\beta > 0$ is a continuous-time stochastic process $\left\{X(t), t \geq 0\right\}$ with the following properties:
\begin{enumerate}
\item $X(0)=0$ with probability one. 
\item $X(t_2)-X(t_1) \sim Gamma(\alpha(t_2)-\alpha(t_1), \beta)$ for $t_1 \leq t_2$. 
\item $X(t)$ has independent increments. 
\end{enumerate}
Recall that a random variable $X$ has a gamma distribution with shape parameter $\alpha > 0$ and scale parameter $\beta > 0$ if its probability density function is given by
\begin{equation} \label{densidadgamma}
f_{\alpha, \beta}(x)=\frac{\beta^{\alpha}}{\Gamma(\alpha)}x^{\alpha-1}\exp^{-\beta x}, \quad x >0,  
\end{equation}
where
\begin{equation*}
\Gamma(\alpha)=\int_{0}^{\infty} t^{\alpha-1}e^{-t}~dt. 
\end{equation*}
We assume in this paper that the system fails when the deterioration level of a degradation process first exceeds a failure threshold $L$. Suppose that a degradation process $X(t)$ starts at time 0 and it grows according to a homogeneous gamma process with parameters $\alpha$ and $\beta$. Let $\sigma_L$ be the first time at which this degradation process exceeds $L$. The random variable $\sigma_L$ is known as the first hitting time distribution and it has the following cumulative distribution function
\begin{equation} \label{survivalL}
F_{\sigma_L}(t) =P(X(t) \geq L)=\int_{L}^{\infty} f_{\alpha t, \beta}(x)~dx=\frac{\Gamma(\alpha t, \beta L)}{\Gamma(\alpha t)}, 
\end{equation}
for $t \geq 0$ where $f_{\alpha t, \beta}$ is given by (\ref{densidadgamma}) and
\begin{equation} \label{incomplete}
\Gamma(\alpha, x)=\int_{x}^{\infty} z^{\alpha-1}e^{-z}~dz, 
\end{equation}
denotes the incomplete gamma function for $x \geq 0$ and $\alpha >0$.

For subsequent analysis, the distribution of $\sigma_{L}-\sigma_{M}$ is used in this paper for two degradation levels $M<L$. According to \cite{CastroMercier}, the survival function of this variable is given by  
\begin{equation}   \label{doblefunction}
\bar{F}_{\sigma_{L}-\sigma_{M}} (t) = \int_{x=0}^{\infty}\int_{y=M}^{\infty} f_{\sigma_{M},X(\sigma_{M})}(x,y)F_{\alpha t, \beta}(L-y) ~dy ~dx,
\end{equation}
where $F_{\alpha t, \beta}$ denotes the distribution function of a gamma distribution with parameters $\alpha t$ and $\beta$ and $f_{\sigma_{M},X(\sigma_{M})}$ denotes the joint density function of $(\sigma_{M}, X(\sigma_{M}))$ provided in \cite{Bertoin} as
\begin{equation*} 
 f_{\sigma_{M}^h,X(\sigma_{M}^h)}(x,y)=\int_{0}^{\infty} \mathbf{1}_{\left\{M \leq y < M+s\right\}} f_{\alpha x, \beta}(y-s) \mu(ds), 
 \end{equation*}
and $\mu(ds)$ denotes the Lévy measure of the
gamma process with parameters $\alpha$ and $\beta$ given by
$$\mu(ds)=\alpha \frac{e^{-\beta s}}{s}, \quad s > 0. $$
Hence, 
\begin{eqnarray*} \nonumber
\bar{F}_{\sigma_{L}-\sigma_{M}} (t) &=& \int_{x=0}^{\infty} ~ dx \int_{s=0}^{\infty}~ ds \int_{y=M}^{M+s} ~dy f_{\alpha x, \beta}(y-s) \mu(ds) F_{\alpha t, \beta}(L-y) ~dy ~dx \\
&=& \int_{x=0}^{\infty} ~dx \int_{s=0}^{\infty} ~ ds \int_{y=M}^{M+s} ~dy \frac{\alpha(y-s)^{\alpha x-1}}{s \Gamma(\alpha x)} \int_{u=0}^{L-y} \frac{u^{\alpha t-1} \exp(-\beta(y+u))\beta^{\alpha x+\alpha t}}{\Gamma(\alpha t)} ~du
\end{eqnarray*}

The model considered above assume that all the degradation processes are independent and identically distributed. However, in some cases, the degradation processes degrade at different rates even though no differences in the environment is present \cite{Xiang}. In the next section, a  {\color{black} random effects process-specific} is used to model such variability.

\subsection{Random effects} \label{random_section}

{\color{black} The random effects model is a useful tool for modeling the variability in the different degradation rates and the heterogeneity among different degradation processes}. In this section we assume that the degradation processes can be described by a gamma process with random effects. {\color{black} To implement this random effects model, we assume that}
the scale parameter of the gamma process is random. It means that both the mean and the variance of the process are affected by the random effect parameter. {\color{black} Some authors have previously studied the random effects model in the gamma process using different probability distributions for the scale parameter. One of the most popular is the gamma distribution since it provides a closed-form expression using a Fisher distribution \cite{Lawless}. 

Inspired by the simplicity of its distribution and by the fact that in Bayesian theory, when no a prior statistical information about the parameters is given uniform distribution is used (\cite{Guida}), a uniform distribution is used to model the {\color{black}inverse of the} scale parameter of the gamma process.

Let $\left\{X_h(t), t \geq 0\right\}$ be a gamma process with random effects $\beta$ that controls the heterogeneity among the different degradation processes. We assume that $\left\{X_h(t), \, t \geq 0\right\}$ is a homogeneous gamma process with parameters $\alpha$ and $\beta$, {\color{black} where $\beta^{-1}$} follows a uniform distribution in $(a,b)$ with $0<a<b<\infty$. Then, the corresponding probability density function of $X_h(t)$ is given by
\begin{eqnarray} \nonumber
f_{X_h(t)} (u) &=& \int_{a}^{b} f_{\alpha t, \beta}(u)\frac{1}{b-a} ~d\beta^{-1} \\ \label{frandomeffect}
&=& {\color{black} \frac{1}{(b-a)} \frac{\Gamma(\alpha t-1, u/b)-\Gamma(\alpha t-1, u/a)}{\Gamma(\alpha t)}}, \quad u \geq 0, 
\end{eqnarray}
where $\Gamma(\cdot,\cdot)$ denotes the upper incomplete gamma function given by Eq. (\ref{incomplete}).

Figure \ref{densidad} plots the density $f_{X_h(t)}(u)$ for $t=5$ for a stationary gamma process with shape parameter $\alpha=1$ and {\color{black} scale parameter $\beta^{-1}$ following a uniform distribution in $(1, b)$. As we can see, the kurtosis of this distribution increases with respect to $b$.  }

\begin{figure}[tbph]
\begin{center} 
\includegraphics[width=0.5\textwidth]{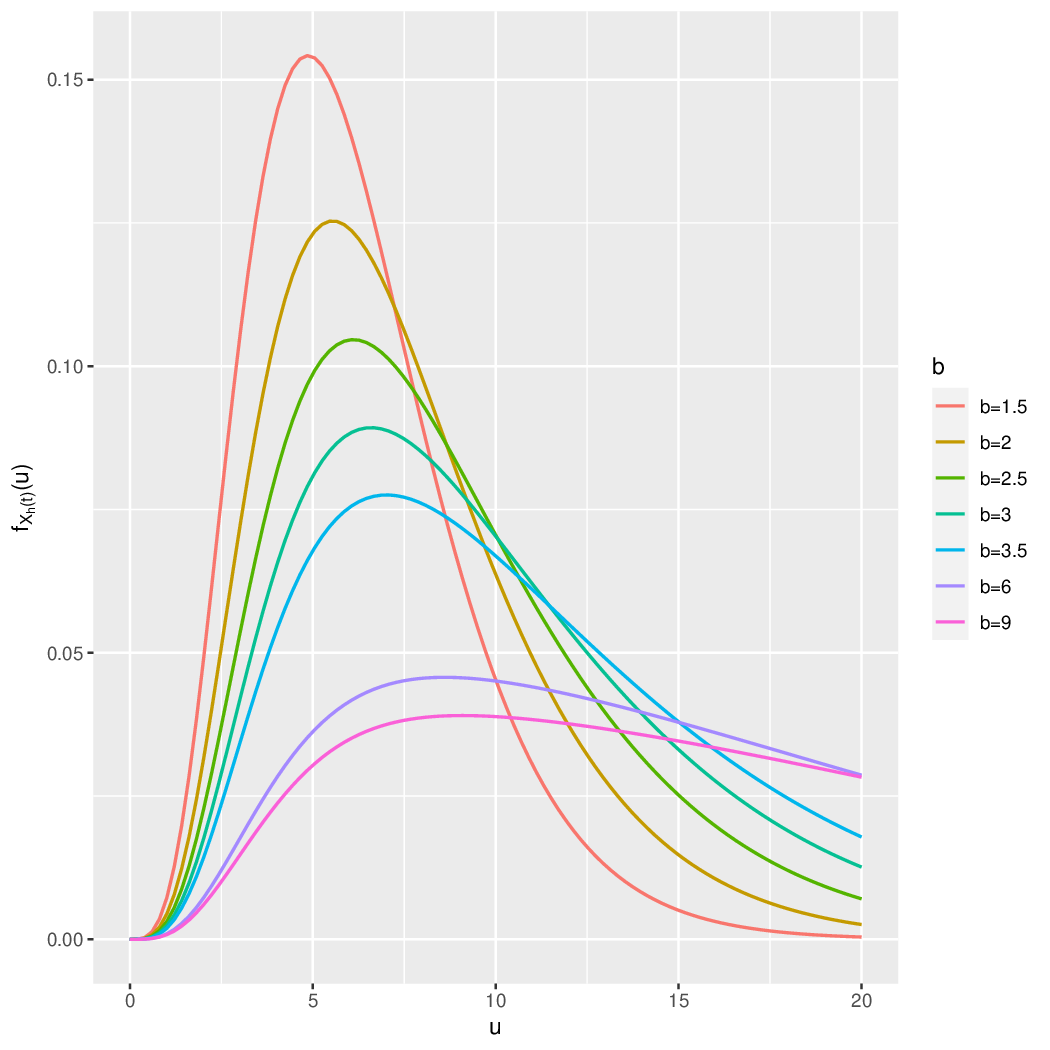}
\caption{{\protect\footnotesize {Density function $f_{X_h(t)}$ for $\alpha=1$, $t=5$ and $\beta^{-1}$ uniform in $(1,b)$.}}} \label{densidad}
\end{center}
\end{figure}

Let $\sigma_L^h$
be the first time in which the process with random effects exceeds the
failure threshold $L$. The survival function of $\sigma_L^h$ can be evaluated as
\begin{eqnarray*}
{F}_{\sigma_L^h}(t) &=& P(X_h(t) \geq L) \\
&=& \int_{L}^{\infty} f_{X_h(t)} (u)~du, 
\end{eqnarray*}
where $f_{X_h(t)} (u)$ is given by (\ref{frandomeffect}). In this random effects model, each process degrades according to a gamma process but the degradation behavior differs from process to process and the resulting process is not longer gamma distributed.

{ \color{black} Moments of $X_h(t)$ can be evaluated as
\begin{eqnarray*}
E[X_h^{n}(t)] &=& \prod_{n^*=0}^{n-1} \left(\alpha t +n^{*}\right) E[1/\beta^n] \\
&=& \prod_{n^*=0}^{n-1} \left(\alpha t +n^{*}\right)\frac{1}{n+1}\sum_{k=0}^{n} a^k b^{n-k}. 
\end{eqnarray*}
Hence, the expectation of $X_h(t)$ is given by 
\begin{eqnarray*}
\mathbb{E}(X_h(t)) &=& \frac{\alpha t (a+b)}{2}, 
\end{eqnarray*}
with variance
\begin{eqnarray*} \nonumber
Var(X_h(t)) &=& \mathbb{E}(X_h(t)^2)-(\mathbb{E}(X_h(t)))^2 \\ \nonumber
&=& \frac{(\alpha t+(\alpha t)^2) (b^2+ab+a^2)}{3}-\frac{(\alpha t)^2 (a+b)^2}{4} \\
&=& \frac{\alpha t(b^2+ab+a^2)}{3}+\frac{(\alpha t)^2}{12}(a-b)^2. 
\end{eqnarray*}
For fixed $t<\infty$, and assuming that $0<a<b<\infty$, $X_h(t)$ has finite mean and variance. 

}
Expectation and variance of the random effects process $\left\{X_h(t), \, t \geq 0\right\}$ increase with the time. As in the case without heterogeneities, expectation increases linearly with the time. In the case of the variance, unlike the case without heterogeneities, variance does not increase linearly with the time. Figures \ref{esperanza} and \ref{varianza2} show the expectation and variance of the process $\left\{X_h(t), \, t \geq 0\right\}$ for a stationary gamma process with shape parameter $\alpha=1$ and $\beta^{-1}$ following a uniform distribution in $(1,b)$.

\begin{figure}
 \centering
  \subfloat[Expectation]{
   \label{esperanza}
    \includegraphics[width=0.5\textwidth]{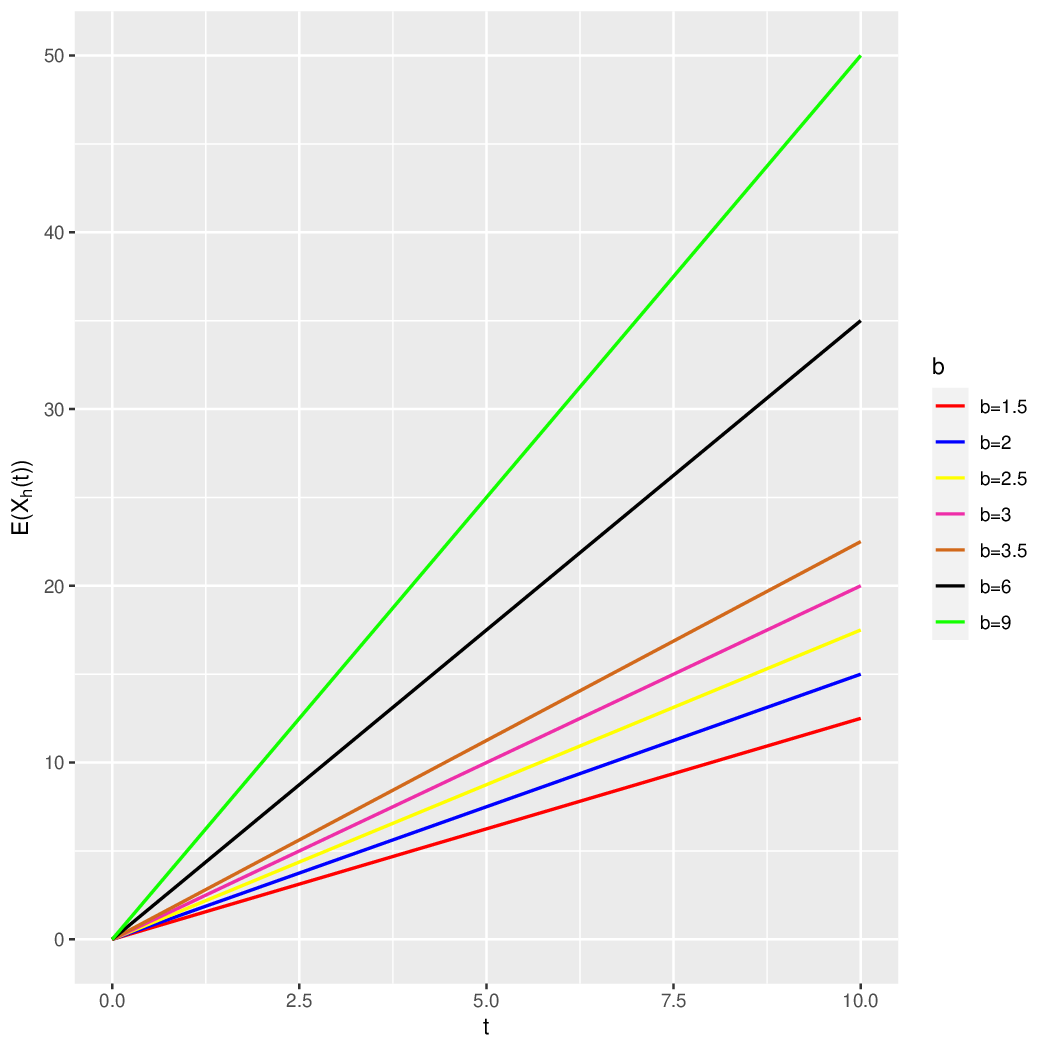}}
  \subfloat[Variance]{
   \label{varianza2}
    \includegraphics[width=0.5\textwidth]{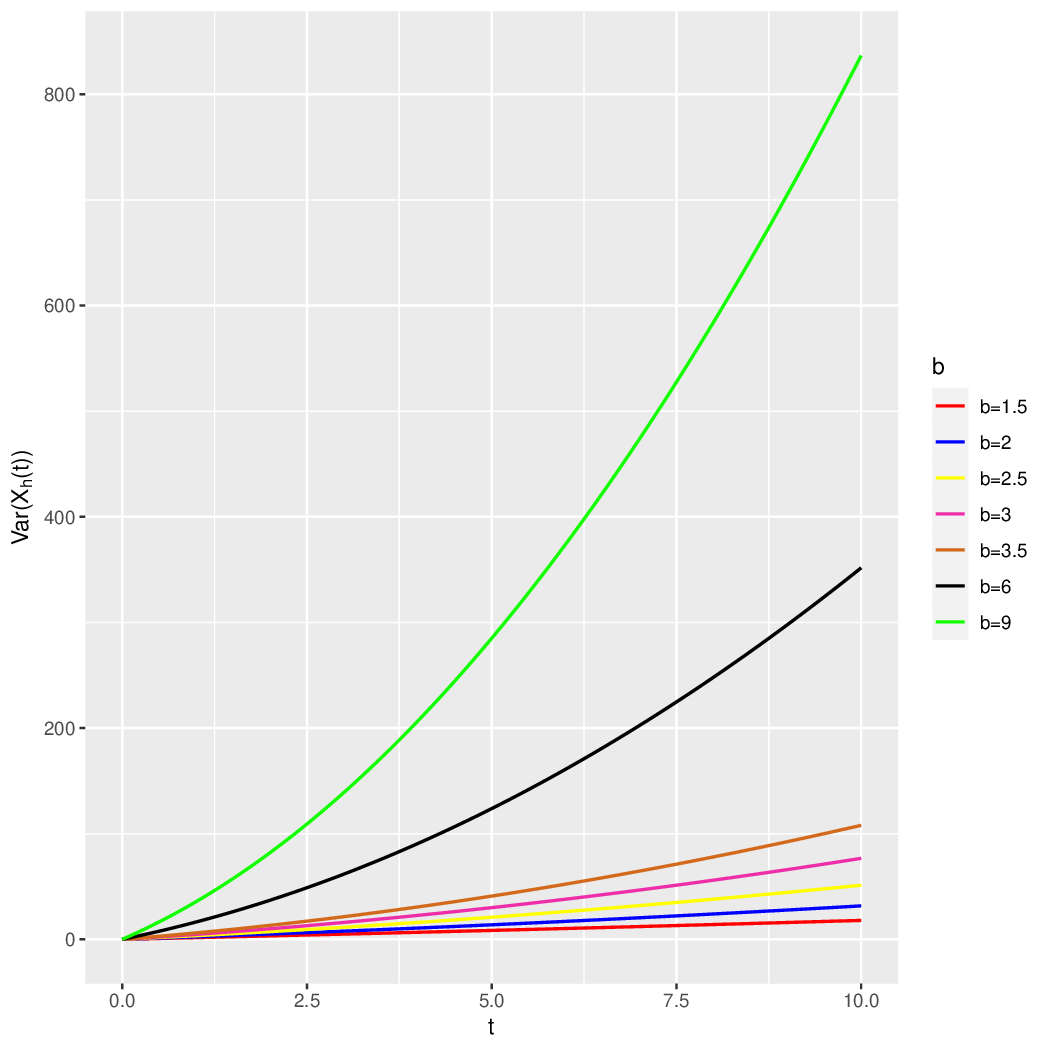}}
 \caption{Expectation and variance of $X_h(t)$ versus $t$ for $\alpha=1$ and {\color{black}$\beta^{-1}$} uniform in $(1,b)$}
 \label{6}
\end{figure}

%
%
%
%
It is well known that, in a gamma process $X(t)$ without heterogeneity, the ratio $\text{Var}(X(t))/\mathbb{E}(X(t))$ does not depend on time. However, in the gamma process with heterogeneity presented in this paper, the quotient 
{\color{black}
\begin{eqnarray} \nonumber
\frac{\text{Var}(X_h(t))}{\mathbb{E}(X_h(t))} &=& \frac{2(1+\alpha t)(b^2+ab+a^2)}{3(a+b)}-\frac{\alpha t(a+b)}{2} \\ \label{q}
&=& \frac{4(b^2+ab+a^2)+\alpha t(b-a)^2}{6(a+b)}, 
\end{eqnarray}
is no longer constant (increases with $t$)}. As \cite{Pulcini} claims, the presence of a noticeable heterogeneity should ensure that the ratio (\ref{q}) increases with the time. Figure \ref{ratio} shows the ratio given by Eq. (\ref{q}) of the process $\left\{X_h(t), \, t \geq 0\right\}$ for a stationary gamma process with shape parameter $\alpha=1$ and {\color{black} $\beta^{-1}$} following a uniform distribution in $(1,b)$. {\color{black} As we can see}, the ratio is increasing with the time. 
\begin{figure}[tp]
\begin{center} 
\includegraphics[width=0.5\textwidth]{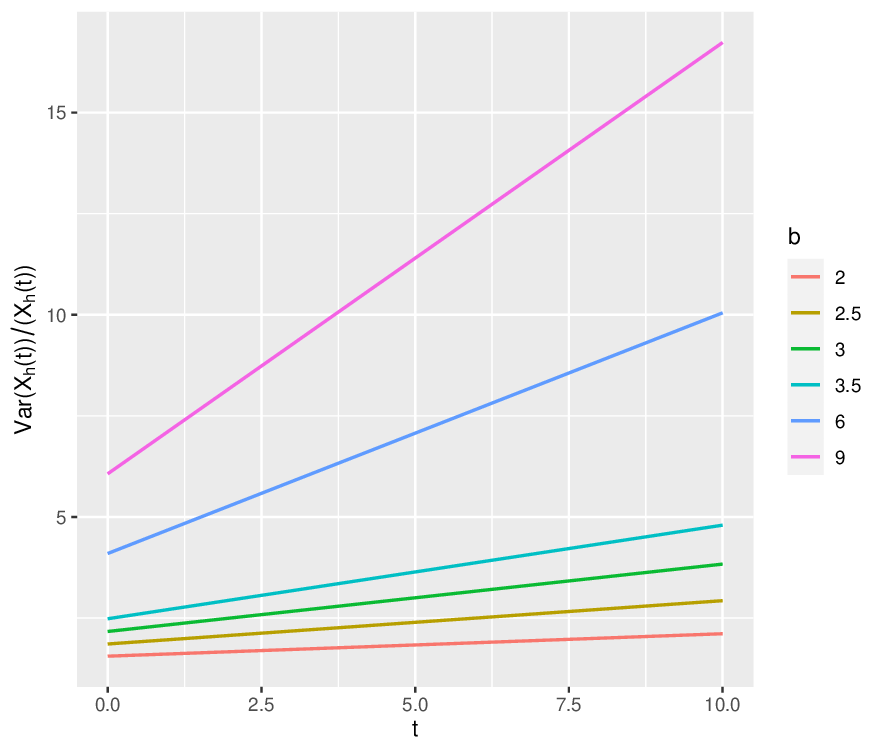}
\caption{{\protect\footnotesize {$Var(X_h(t))/E(X_h(t))$ versus $t$ with $\beta^{-1}$ uniform in $(1,b)$}}} \label{ratio}
\end{center}
\end{figure}

{\color{black}
The random effects model with heterogeneity modelled using a uniform distribution is next compared to a random effects model with heterogeneity modelled using a gamma distribution and compared to a model without heterogeneity. For that, let $\left\{X_h^{(1)}(t), \, t \geq 0 \right\}$, $\left\{X_h^{(2)}(t), \, t \geq 0 \right\}$ and $\left\{X(t), \, t \geq 0 \right\}$ be three stochastic processes where $X_h^{(1)}(t)$ denotes the degradation level at time $t$ for a gamma process with shape parameter $\alpha$ and scale parameter $\beta_1$ where $\beta_1^{-1}$ follows a uniform distribution $U(a,b)$, $X_h^{(2)}(t)$ denotes the degradation level at time $t$ for a gamma process with shape parameter $\alpha$ and scale parameter $\beta_2$ where $\beta_2$ follows a gamma distribution with parameter $k_1$ and $k_2$ $Gamma(k_1,k_2)$ and $X(t)$ denotes the degradation level at time $t$ for a gamma process with shape parameter $\alpha$ and scale parameter $\beta_3$ with deterministic $\beta_3$. To compare the three processes, firstly the expectations are standardized imposing 

\begin{equation*} 
E[X_h^{(1)}(t)]=E[X_h^{(2)}(t)]=E[X(t)], \quad t \geq 0. 
\end{equation*}
To get the same expectation, the parameters of the three processes fulfill
\begin{equation} \label{expectation3}
\frac{1}{\beta_3}=\frac{a+b}{2}=\frac{k_1}{k_2}. 
\end{equation}
Since the three processes have the same expectation, we get that the comparison in terms of variances is reduced to the comparison of the moments of second order $\mathbb{E}[X_h^{(1)}(t)^2]$, $\mathbb{E}[X_h^{(2)}(t)^2]$ and $\mathbb{E}[X(t)^2]$. We get that
\begin{eqnarray*}
\mathbb{E}[X_h^{(1)}(t)^2] &=& \left(\alpha t+(\alpha t)^2\right)\mathbb{E}\left[\frac{1}{\beta_1^2}\right]=\left(\alpha t+(\alpha t)^2\right)\frac{b^2+ab+a^2}{3} \\
\mathbb{E}[X_h^{(2)}(t)^2] &=& \left(\alpha t+(\alpha t)^2\right)\mathbb{E}\left[\frac{1}{\beta_2^2}\right] =\left(\alpha t+(\alpha t)^2\right)\frac{k_1(k_1+1)}{k_2}\\
\mathbb{E}[X(t)^2] &=& \left(\alpha t+(\alpha t)^2\right)\mathbb{E}\left[\frac{1}{\beta_3^2}\right] 
\end{eqnarray*}
Comparing $X_h^{(1)}(t)$ and $X_h^{(2)}(t)$ with parameters fulfilling Eq. (\ref{expectation3}) in terms of variances, we get that
$$Var(X_h^{(1)}(t)) \geq (\leq) \, Var(X_h^{(2)}(t)) \Leftrightarrow \frac{2(b^2+ab+a^2)}{3(a+b)} -1 \geq (\leq) \,  k_1.   $$
Hence, the variances of the two processes $Var(X_h^{(1)}(t))$ and $Var(X_h^{(2)}(t))$ depend on the parameters of the processes. Comparing the random effect model with uniform distribution and the process without heterogeneity, we get that
\begin{eqnarray*}
\mathbb{E}[X_h^{(1)}(t)^2]  &=&  \left(\alpha t+(\alpha t)^2\right)\frac{b^2+ab+a^2}{3}  \\
\mathbb{E}[X(t)^2]  &=&  \left(\alpha t+(\alpha t)^2\right) \frac{(a+b)^2}{4}
\end{eqnarray*}
hence $Var(X_h^{(1}(t)) \geq Var(X(t))$ for all $t$ and parameters fulfilling Eq. (\ref{expectation3}). 
Hence for two processes with the same expectation, a gamma process without heterogeneity shows a lower variance with respect to a gamma process with random effects model under a uniform distribution. 

}

Expression (\ref{frandomeffect}) allows us to obtain the likelihood function for observed data. Suppose now that $X_h(t)$ is recorded at times $t_1, t_2, \ldots, t_n$ with observations $x_1, x_2, \ldots, x_n$. The $x$-increments are defined by $\Delta x_j=x_{j}-x_{j-1}$ for $j=1, 2, \ldots, n$ and $x_0=0$ and the $t$-increments are defined by $\Delta t_j=t_{j}-t_{j-1}$ with $t_0=0$. Using (\ref{frandomeffect}) for fixed $\alpha$ and for a realization of $\beta^{-1}$, the joint density of $\left(\Delta x_1, \Delta x_2, \ldots, \Delta x_n\right)$ is given by the product

\begin{equation} \label{densidadre}
f(\Delta x_1, \Delta x_2, \ldots, \Delta x_n)=\frac{\prod_{j=1}^{n}(\Delta x_j)^{\alpha \Delta t_j-1}}{\prod_{j=1}^{n} \Gamma(\alpha \Delta t_j)} \beta^{\alpha t_n} \exp(-\beta x_n). 
\end{equation}

Suppose that $m$ degradation processes $\left\{X_{h,i}(t), t \geq 0\right\}$ for $i=1, 2, \ldots, m$ are observed with the $i$-th process observed at times $t_{0,i} < t_{1,i}, \ldots, <t_{n_i,i}$ with values
$x_{0,i} < x_{1,i}, \ldots, <x_{n_i, i}$. Analogously to \cite{Lawless}, the joint density of the increments 
$\Delta x_{j, i}=x_{j, i}-x_{j-1, i}$ for the $i$ process is given by (\ref{densidadre}) replacing $\Delta x_{j}$ by $\Delta x_{j, i}$, $\Delta t_j$ by $\Delta t_{j, i}$. The product across $i=1, 2, \ldots, m$ and the integration with respect to $\beta$ gives the likelihood function of the observed data in the $m$ processes.

\begin{equation} \label{densidadre2}
L(\alpha, a, b)=\frac{1}{(b-a)^m}\prod_{i=1}^{m}\frac{\prod_{j=1}^{n_i}(\Delta x_{j,i})^{\alpha \Delta t_{j,i}-1}}{\prod_{j=1}^{n_i} \Gamma(\alpha \Delta t_{j,i})} \frac{\Gamma(\alpha t_{n,i}-1,x_{n,i}/b)-\Gamma(\alpha t_{n,i}-1,x_{n,i}/a)}{x_{n,i}^{\alpha t_{n,i}}}.
\end{equation}

Inference about $(\alpha, a, b)$ can be based on the usual maximum likelihood procedures using Eq.(\ref{densidadre2}).

\begin{example} Different degradation processes have been simulated from $t=0$ to $t=30$ for $\alpha=1.5$ and $\beta^{-1}$ uniform in $(1-0.3, 1+0.3)$. The trajectories are plotted in Figures \ref{effects} and \ref{effects2} for six degradation processes and twenty six degradation processes respectively.

For fixed $\alpha=1.5$ and $\beta^{-1}$ uniform in $(1-\alpha^*, 1+\alpha^*)$, Figures \ref{likelihood} and \ref{likelihood2} show the negative of the logarithm of the likelihood of the six simulation processes versus $\alpha^*$. As we can check, the minimum of this function is reached for 
$\alpha^*=0.39$ with a value of -473.5824 (for 6 simulated processes) and
$\alpha^*=0.302$ with a value of the negative of the likelihood of -431.5405 for 26 degradation processes.

%
%
%
%
%
%
%

\begin{figure}
 \centering
  \subfloat[Trajectories]{
   \label{effects}
    \includegraphics[width=0.5\textwidth]{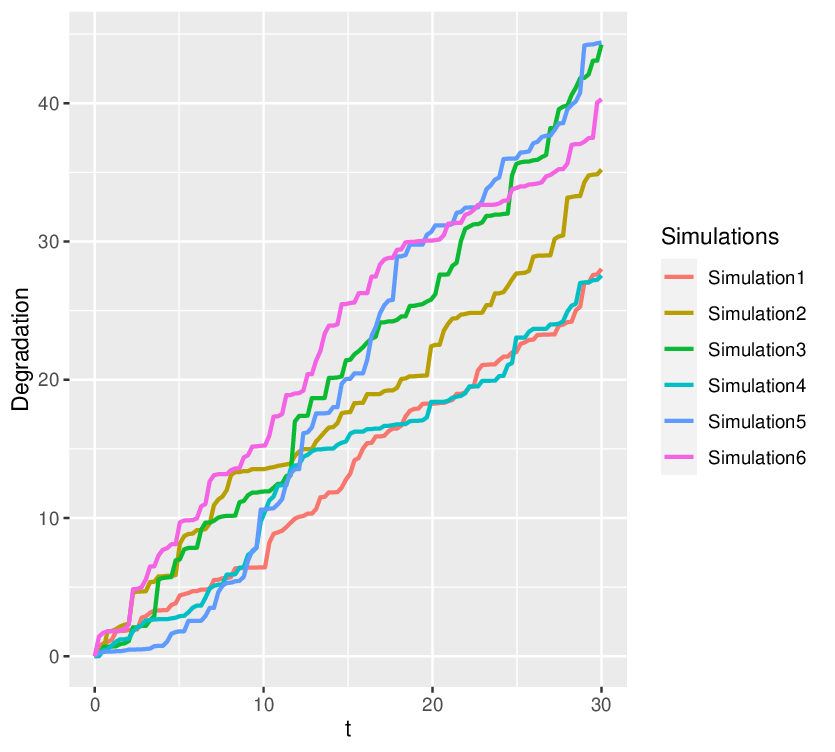}}
  \subfloat[Negative of the log likelihood]{
   \label{likelihood}
    \includegraphics[width=0.5\textwidth]{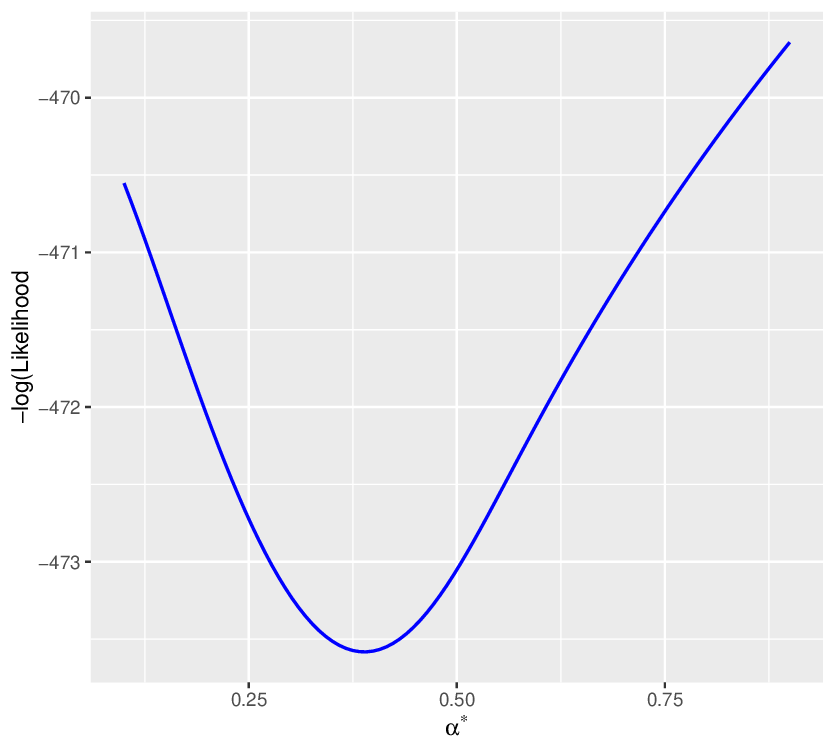}}
 \caption{Trajectories and log likelihood}
 \label{6}
\end{figure}

\begin{figure}
 \centering
  \subfloat[Trajectories]{
   \label{effects2}
    \includegraphics[width=0.5\textwidth]{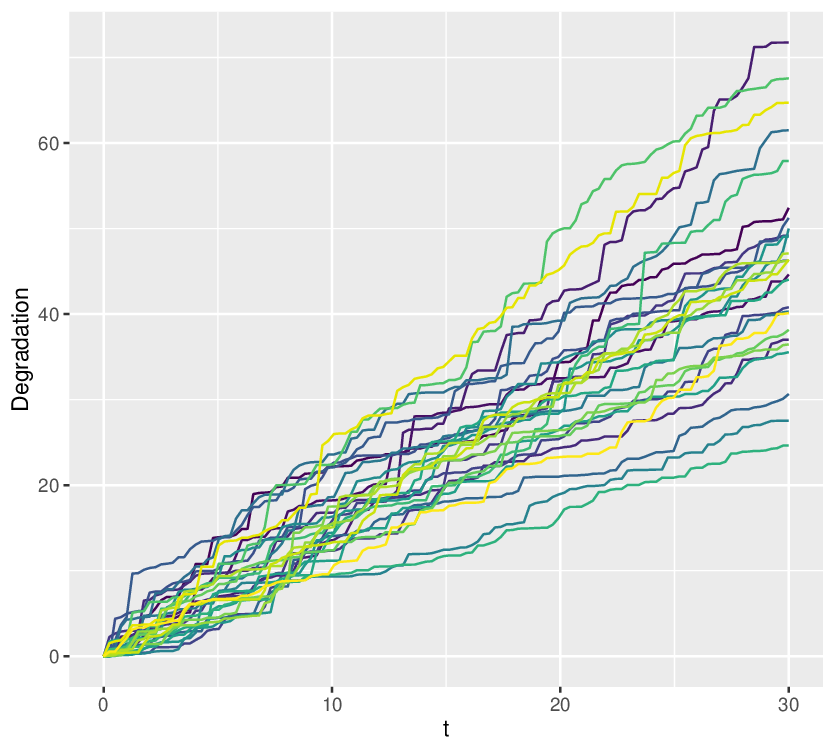}}
  \subfloat[Negative of the log likelihood]{
   \label{likelihood2}
    \includegraphics[width=0.5\textwidth]{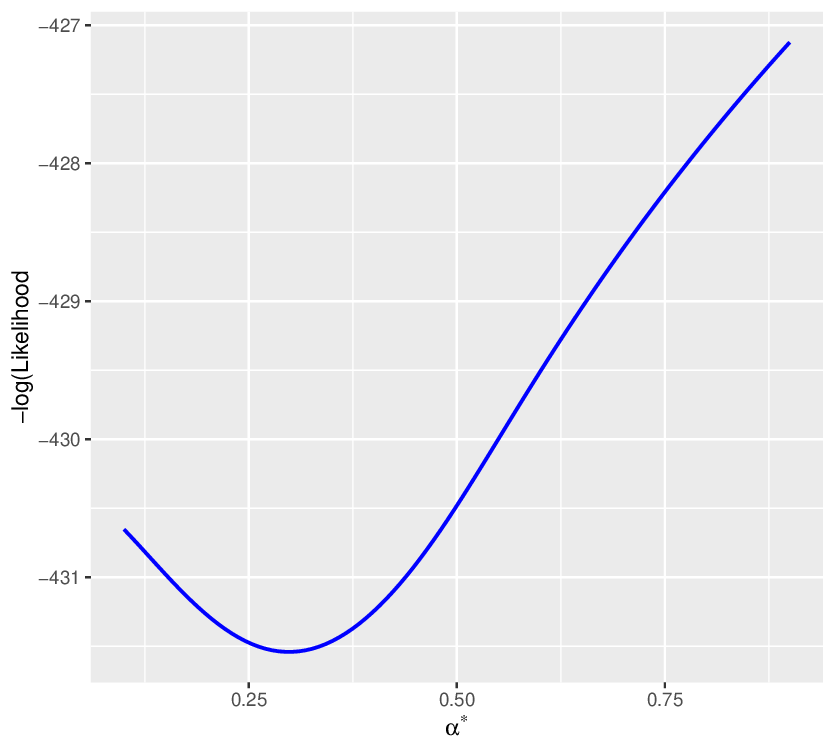}}
 \caption{Trajectories and log likelihood}
 \label{26}
\end{figure}

%

\end{example}


}

\section{Combined process of arrival and growth}  \label{c}

{\color{black}

Since the shot noise process with intensity given by Eq. (\ref{snprocess}) is a Cox process, this section shows some previous well-known results about general Cox processes. These results are used to evaluate quantities of interest of the model. 
}
The degradation processes start at random times $S_1, S_2, \ldots$ following a shot noise Cox process with intensity given by (\ref{snprocess}) and then they grow independently following a homogeneous gamma process with parameters $\alpha$ and $\beta$. Next, we shall analyse the combined process of arrival and growth. 

We define the deterioration level of the degradation process $k$ at time $t$ as 
\begin{equation*}
X_k(t)=X^{(k)}(t-S_k), \quad t \geq S_k,
\end{equation*}
where $X^{(k)}$ denotes a homogeneous gamma process with parameters $\alpha$ and $\beta$. We set $X^{(i)}$ with $i \in N$ to be i.i.d of a stationary gamma process with parameters $\alpha$ and $\beta$. 

Let $W_k$ be the time where the deterioration level of the degradation process $k$ exceeds the threshold $L$. Hence, $W_k$ is given by
\begin{equation*}
 W_k=S_k+\sigma_L,  \ \ \text{for} \ k=1,2, \ldots
 \end{equation*}
Figure \ref{realization} shows a realization of the combined process.

\begin{figure}[tbph]
\begin{center} 
\includegraphics[width=0.7\textwidth]{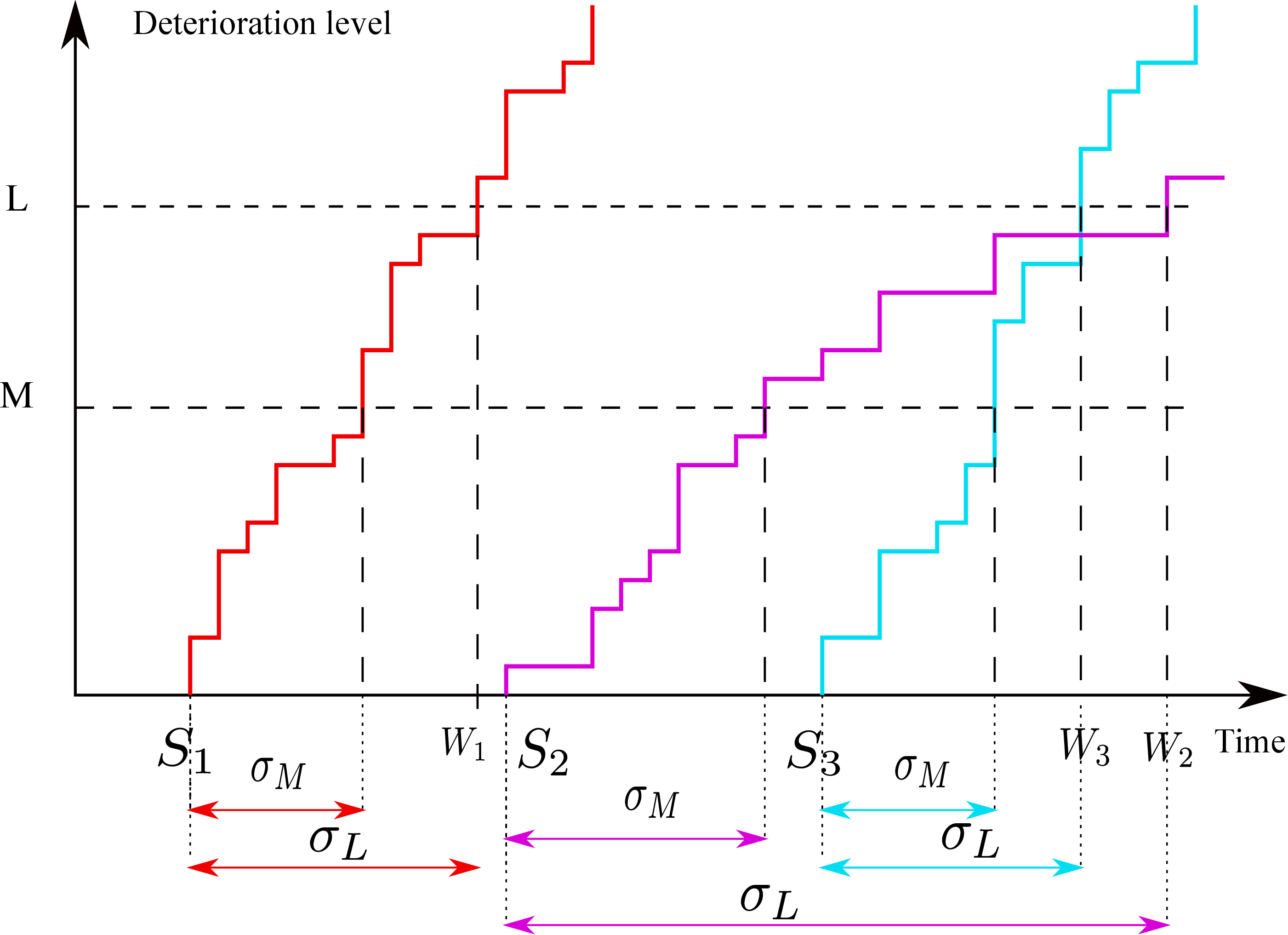}
\caption{{\protect\footnotesize {Realization of the combined process.}}} \label{realization}
\end{center}
\end{figure}

Let $N_L(t)$ be the  number of degradation processes whose deterioration levels exceed the failure threshold $L$ at time $t$, that is, 
\begin{equation} \label{NL}
N_L(t)=\sum_{k=1}^{\infty} \mathbf{1}_{\left\{W_k \leq t\right\}}. 
\end{equation}
The distribution of the counting process $\left\{N_L(t), \, t \geq 0\right\}$ is next obtained. 

Before obtaining the distribution of $\left\{N_L(t), \, t \geq 0\right\}$, some definitions and properties of the Cox processes are first recalled. These properties allow to obtain the combined expression for arrival and growth of the degradation processes. 

Because Cox process are essentially Poisson processes, most results for Poisson processes have counterparts for Cox processes using an expectation to take into account the stochastic intensity. The preservation of basic point process operations is an example of these results. We focus on the random translation of a point process, the so-called displaced process. 

\begin{definition} Let $\left\{N^*(t), t \geq 0\right\}$ be a counting process with occurrence times $S_1, S_2, \ldots$ and suppose that $\left\{D_1, D_2, D_3, \ldots\right\}, \quad i=1, 2, \ldots$ is a sequence of independent and identically distributed, non-negative random variables independent of $S_i$. The point process $\left\{N(t), t \geq 0\right\}$ whose occurrence times are $\left\{S_i+D_i\right\}$ for $i=1, 2, \ldots$ is called a displaced process.  
\end{definition}
It is well known that if $\left\{N^*(t), t \geq 0\right\}$ follows a non-homogeneous Poisson process with intensity $\lambda^*(t)$, then the displaced process $\left\{N(t), t \geq 0\right\}$ with occurrence times given by $\left\{T_i+D_i\right\}$ is also a non homogeneous Poisson process \cite{Caballe} with intensity
\begin{equation*}
\lambda(t)=\int_{0}^{t} \lambda^*(u)f(t-u) ~du, \quad t \geq 0, 
\end{equation*}
where $f$ denotes the common density of the random variables $D_i$. 

This preservation of the random translation of a non homogeneus Poisson process has the counterpart for the Cox process. It is explained in the following result (see References \cite{Serfozo}and \cite{Moller} for more details).

\begin{lemma} \label{lemaCox}
If $\left\{N^*t), t \geq 0\right\}$ is a Cox process with stochastic intensity $\lambda^*(t)$, then the displaced process $\left\{N(t), t \geq 0\right\}$ with occurrence times given by $\left\{T_i+D_i\right\}$ is also a Cox process with stochastic intensity
\begin{equation*} \label{stochasticintensity}
\lambda(t)=\int_{0}^{t} \lambda^*(u)f(t-u) ~du, \quad t \geq 0,  
\end{equation*}
where $f$ denotes the common density of the random variables $D_i$. 
\end{lemma}

Since, in this paper,  the arrival process (the shot noise Cox process) is a Cox process, using Lemma \ref{lemaCox}, the mathematical model for the counting process $\left\{N_L(t), \, t \geq 0\right\}$ given by Eq. (\ref{NL}) is obtained and it is shown in the next result.

\begin{lemma} \label{lemaa}
Let $S_1, S_2, \ldots$ be the occurrence times of the shot noise Cox process with stochastic intensity given by (\ref{snprocess}), then the displaced process $\left\{N_L(t), t \geq 0\right\}$ with occurrence times given by $\left\{S_i+\sigma_L\right\}$ is also a Cox process with stochastic intensity
\begin{equation} \label{stochasticintensityL}
\lambda_L(t)=\int_{0}^{t} \lambda^*(u)f_{\sigma_L}(t-u) ~du, \quad t \geq 0,  
\end{equation}
where $f_{\sigma_L}$ denotes the density of the first hitting time $\sigma_L$ with distribution given by Eq. (\ref{survivalL}). 
\end{lemma}

The expectation of the stochastic intensity $\lambda_L(t)$ of the combined process of arrivals and growth is obtained.

\begin{lemma} \label{lemasi}
The expectation of the stochastic intensity of the displaced process $\left\{N_L(t), t \geq 0\right\}$ with occurrence times given by $\left\{S_i+\sigma_L\right\}$ and where $S_i$ are given by a shot noise Cox process with intensity given by Eq. (\ref{snprocess}) is given by
\begin{equation*} \label{combinedintensity}
\mathbb{E}[\lambda_L(t)]=\lambda_0F_{\sigma_L}(t)+\mu \int_{0}^{t}H(u)f_{\sigma_L}(t-u) ~du, 
\end{equation*}
where $f_{\sigma_L}$ is the density function of the hitting time $\sigma_L$ and $H(t)$ is given by
$$H(t)=\int_{0}^{t} e^{-\delta u} ~du. $$
\end{lemma}

\begin{proof}
As we explained before, $\left\{N_L(t), t \geq 0\right\}$ is a Cox process with intensity $\lambda_L(t)$ given by (\ref{stochasticintensityL}) with expectation 
$$\mathbb{E}[\lambda_L(t)]=\int_{0}^{t} \mathbb{E}[\lambda^*(u)]f_{\sigma_L}(t-u) ~du, $$
and, using $\mathbb{E}[\lambda^*(u)]$ given in (\ref{expectationlambdastar}), the result is obtained {\color{black}using Fubini theorem since $\lambda^*$ is integrable}. 
\end{proof}

Finally, Lemma \ref{lemasi} allows us to obtain the expected number of degradation processes that exceed the failure threshold $L$ at time $t$ given by Eq. (\ref{NL}). 
To derive $\mathbb{E}[N_L(t)]$, Lemma \ref{lemaNumberArrivals} is used.
\begin{eqnarray*}\label{expected_arrivals_displaced}
\mathbf{E}\left[N_L(t)\right] = \int_{0}^{t} \mathbb{E}[\lambda_L(s)] ~ds=\int_{0}^{t} \mathbb{E}[\lambda^*(u)]F_{\sigma_{L}}(t-u)~du, 
\end{eqnarray*}
where $\mathbb{E}[\lambda^*(u)]$ is given by (\ref{expectationlambdastar}). 

Notice that Lemmas \ref{lemaa} and \ref{lemasi} have a counterpart in the random effects case replacing $f_{\sigma_L}$ by $f_{\sigma_L^h}$ which is the density probability function of $\sigma_L^h$.

{\color{black} Next section analyzes the failure time for this system. The failure time distribution is obtained and it is shown that it is increasing failure rate.}

\section{Time to the system failure}  \label{d}

The system can be considered to be failed if at least one of the degradation processes reaches its failure threshold. We assume that the failure threshold is $L$ for all the degradation processes. 

Let $W_{[1]}$ be the instant at which, for the first time, the deterioration level of a degradation process exceeds the {\color{black} failure} threshold L, 
$$W_{[1]}=\min_{i=1, 2, \ldots} \left\{W_i\right\}. $$
The probability distribution of $W_{[1]}$ is next obtained. 

\begin{lemma} \label{lemaW1} Let $W_{[1]}$ be the instant at which, for the first time, the deterioration level of a degradation process exceeds the failure threshold $L$. Then, the survival function of $W_{[1]}$ is given by
\begin{equation*}
\bar{F}_{W_{[1]}}(t)=\exp\left(-\lambda_0 \int_{0}^{t}F_{\sigma_L}(u)~ du -\mu \int_{0}^{t}   \left(1-\exp \left(-\int_{0}^{u}e^{-\delta w}F_{\sigma_L}(u-w) ~dw\right) \right)~du\right)
\end{equation*}
where $F_{\sigma_L}$ is given by (\ref{survivalL}). 
\end{lemma}

\begin{proof}
As we obtained above, counting process $\left\{N_L(t), t \geq 0\right\}$ is a Cox process with intensity $\lambda_L(t)$ given by (\ref{stochasticintensityL}). Hence
 \begin{eqnarray*} \nonumber
 \bar{F}_{W_{[1]}}(t) &=& P(N_L(t)=0) \\ \nonumber
&=& \mathbb{E}\left[\exp\left\{-\int_{0}^{t} \lambda_L(u) ~du\right\} \right]\\ \nonumber
&=& \mathbb{E}\left[\exp\left\{-\int_{0}^{t} \lambda^*(u)F_{\sigma_L}(t-u)~du \right\}\right] \\ 
&=& \exp\left\{-\lambda_0 \int_{0}^{t}F_{\sigma_L}(u)~ du\right\}\mathbb{E}\left[\exp\left\{-\int_{0}^{t}\sum_{i=1}^{N(t)}e^{-\delta(u-T_i)}F_{\sigma_L}(t-u)\mathbf{1}_{\left\{u >T_i\right\}}~  du \right\}\right]. 
 \end{eqnarray*}
Given $N(t)=n$ with $n>0$, we get that
\begin{eqnarray*} \nonumber
P\left(W_{[1]}>t|N(t)=n\right) &=& C_1(t)\mathbb{E}\left[\prod_{i=1}^{n} \exp\left\{-\int_{0}^{t} e^{-\delta(u-T_i)}F_{\sigma_L}(t-u) \mathbf{1}_{\left\{u >T_i\right\}}~ du\right\}\right] \\
 &=& C_1(t)\mathbb{E}\left[\left(\exp\left\{-\int_{0}^{t} e^{-\delta(u-U)}F_{\sigma_L}(t-u) \mathbf{1}_{\left\{u >U\right\}}~ du\right\}\right)\right]^n, 
\end{eqnarray*}
where
\begin{equation} \label{C1}
C_1(t)=\exp\left(-\lambda_0 \int_{0}^{t}F_{\sigma_L}(u)~ du\right), \quad t \geq 0, 
\end{equation}
and $U$ is a uniform variable in the interval $(0,t)$. Then 
\begin{eqnarray*}
P\left(W_{[1]}>t|N(t)=n\right) &=& C_1(t)\left(\int_{s=0}^{t} \frac{1}{t} \exp\left\{-\int_{s}^{t}e^{-\delta(u-s)}F_{\sigma_L}(t-u)~ du \right\}~ ds\right)^n, 
\end{eqnarray*}
therefore
\begin{eqnarray} \nonumber
\bar{F}_{W_{[1]}}(t) &=& \sum_{n=0}^{\infty} P(W_{[1]}>t|N(t)=n)P(N(t=n)) \\ \nonumber
&=& C_1(t) \exp(-\mu t)+ \sum_{n=1}^{\infty} P(W_{[1]}>t)P(N(t=n)) \\ \nonumber
&=& C_1(t) \exp(-\mu t)+C_1(t) \exp(-\mu t) \sum_{n=1}^{\infty} \frac{A(t)^n}{t^n}\frac{(\mu t)^n}{n!} 
\end{eqnarray}
where $A(t)$ is given by
\begin{equation*} 
A(t)=\int_{0}^{t} \exp\left(-\int_{0}^{x}e^{-\delta w}F_{\sigma_L}(x-w)~ dw\right)~ dx.  
\end{equation*}
Hence
\begin{eqnarray*} \nonumber
\bar{F}_{W_{[1]}}(t) &=& C_1(t) \exp(-\mu t)+C_1(t) \exp(-\mu t) \left(\exp(A(t) \mu)-1\right) \\ \nonumber
&=& C_1(t) \exp(-\mu t)) \exp(A(t) \mu \\
&=& C_1(t)C_2(t)
\end{eqnarray*}
where $C_1(t)$ is given by (\ref{C1})
and $C_2(t)$ equals to 
\begin{equation} \label{C2}
C_2(t)=\exp\left(-\mu \int_{0}^{t} \left(1-\exp\left(-\int_{0}^{x} e^{-\delta w}F_{\sigma_L}(x-w)~ dw \right)\right)~ dx\right), 
\end{equation}
and the result holds. 
\end{proof}

\begin{lemma} \label{lemmaIFR} Let $W_{[1]}$ be the instant at which, for the first time, the deterioration level of a degradation process exceeds the failure threshold $L$. Then $W_{[1]}$ is increasing failure rate (IFR). 
\end{lemma}

\begin{proof}
Using Lemma \ref{lemaW1}, the failure rate function of {\color{black} $W_{[1]}$} is given by
\begin{equation*}
r_{W_{[1]}}(t)=\frac{-C_1'(t)}{C_1(t)}+\frac{-C_2'(t)}{C_2(t)}, 
\end{equation*}
where $C_1$ and $C_2$ are given by (\ref{C1}) and (\ref{C2}) respectively. {\color{black} By making certain calculations}, we get that
\begin{equation*}
r_{W_{[1]}}(t)=\lambda_0F_{\sigma_L}(t)+\mu \left(1-\exp\left(-\int_{0}^{t}e^{-\delta w}F_{\sigma_L}(t-w) ~ dw\right)\right), 
\end{equation*}
with derivative
\begin{equation*}
r_{W_{[1]}}'(t)=\lambda_0 f_{\sigma_L}(t)+\mu\exp\left(-\int_{0}^{t} \exp^{-\delta w}F_{\sigma_L}(t-w)~ dw \right) \left(\int_{0}^{t}e^{-\delta w} f_{\sigma_L}(t-w)~ dw \right)
\end{equation*}
and it is non-decreasing in $t$ independent of the monotony of the failure rate of $\sigma_L$.  
\end{proof}
Notice that Lemmas \ref{lemaW1} and \ref{lemmaIFR} have a counterpart replacing $F_{\sigma_L}$ by $F_{\sigma_L^h}$. 

The limit of $r_{W_{[1]}}(t)$ is given by
\begin{equation*}
\lim_{t \rightarrow \infty} r_{W_{[1]}}(t)=\lambda_0+\mu\left(1-\exp(-1/\delta)\right). 
\end{equation*}

The increasing failure rate implies that the preventive maintenance is potentially worth implementing to improve the system reliability.

\section{Maintenance policy}  \label{e}
It is assumed that the degradation of the degradation processes cannot be directly and instantaneously observed. The system is monitored through periodic inspections which reveal the exact state of the degradation processes without error. {\color{black} Failures are {\em non self-announcing}, so that system failures are only detected  through inspections.} 
This paper makes the following assumptions. 

\begin{enumerate}
\item The system is inspected each $T$ time units and the following decisions are taken. 
\begin{itemize}
\item If the degradation level of a degradation process exceeds a preventive maintenance threshold $M$ but the system is not failed, a preventive replacement is performed and the system is replaced by a new one. 
\item If the system is failed, a corrective replacement is performed and the system is replaced by a new one. 
\item If there is no degradation processes in the system or their degradation levels do not exceed the preventive maintenance threshold, the system is left as it is. 
\end{itemize}
\item Replacements
are immediately performed and the duration time of these replacements is negligible. 
\item A sequence of costs is associated with the different maintenance actions:
\begin{itemize} 
\item A cost of $C_c$ monetary units is incurred when a corrective replacement is performed. 
\item A cost of $C_p$ monetary units is incurred when a preventive replacement is performed. 
\item The cost of the inspections is equal to $C_I$ monetary units. 
\item The downtime cost is equal to $C_d$ monetary units per {\color{black} time units}. 
\end{itemize}
\end{enumerate}

First, let $\left\{N_M(t), t \geq 0\right\}$ be the number of degradation processes that exceed the preventive threshold $M$ at time $t$. That is, 
\begin{equation*}
N_M(t)=\sum_{i=1}^{\infty} \mathbf{1}_{\left\{S_i+\sigma_M \leq t\right\}}. 
\end{equation*}
Using a similar reasoning to the development of $\left\{N_L(t), t \geq 0\right\}$, the point process $\left\{N_M(t), t \geq 0\right\}$ follows a Cox process with stochastic intensity
\begin{equation} \label{lambdaM2}
\lambda_M(t)=\int_{0}^{t} \lambda^*(u)f_{\sigma_M}(t-u)~du, \quad t \geq 0, 
\end{equation}
where $\lambda^*(u)$ is given by (\ref{snprocess})

Let $V_{[1]}$ be the instant at which, for the first time, the deterioration level of a degradation process exceeds the preventive threshold $M$. Using a similar reasoning to the method developed above, the distribution of $V_{[1]}$ can be expressed as 
\begin{eqnarray*}
\bar{F}_{V_{[1]}}(t) &=& \mathbb{E}\left[\exp\left\{-\int_{0}^{t} F_{\sigma_M}(t-u) \lambda^*(u)~ du \right\}\right], \quad t \geq 0, 
\end{eqnarray*}
where $F_{\sigma_M}$ is given by Eq. (\ref{survivalL}) replacing $L$ by $M$.

Let $R$ be the time to a replacement, that is, the time to a preventive replacement or a corrective maintenance. Hence, 
\begin{equation*}
R=\sum_{k=0}^{\infty} (k+1)T \mathbf{1}_{\left\{kT \leq V_{[1]} \leq (k+1)T\right\}}, 
\end{equation*}
with expectation
\begin{eqnarray} \label{expectedlength}
\mathbb{E}[R] &=&  \sum_{k=0}^{\infty} (k+1)T P(kT \leq V_{[1]} \leq (k+1)T) \\ \nonumber
&=& T \sum_{i=0}^{\infty} \bar{F}_{V_{[1]}}(iT). 
\end{eqnarray}
Let $N_I(R)$ be the total number of inspections performed on the system in a replacement cycle, then
\begin{equation} \label{expectedinspection}
\mathbb{E}[N_I(R)]=\frac{\mathbb{E}[R]}{T}=\sum_{i=0}^{\infty} \bar{F}_{V_{[1]}}(iT). 
\end{equation}
\subsection*{Preventive maintenance probability}
A preventive replacement is performed at time $(k+1)T$ if the system is not failed at time $(k+1)T$ and the preventive threshold $M$ is exceeded for the first time in $(kT, (k+1)T]$, that is, 
$$\left\{kT \leq V_{[1]} \leq (k+1)T < W_{[1]}\right\}.  $$

Let $\left\{N_{L-M}(t), t \geq 0\right\}$ be the following counting process
\begin{equation} \label{Nstar}
N_{L-M}(t)=\sum_{n=0}^{\infty} \mathbf{1}_{\left\{V_i+(\sigma_L-\sigma_M) \leq t\right\}},  
\end{equation}
where the distribution of $\sigma_L-\sigma_M$ is given by Eq. (\ref{doblefunction}). Since $\left\{V_i \right\}$,  for $i=1, 2, \ldots$ is a Cox process with intensity $\lambda_M$ given by (\ref{lambdaM2}), the translated process $N_{L-M}(t)$ given in (\ref{Nstar}) is a Cox process with intensity 
\begin{equation*}
\lambda_{L-M}(t)=\int_{0}^{t} \lambda_M(u)f_{\sigma_L-\sigma_M}(t-u) ~du, 
\end{equation*}
where $f_{\sigma_L-\sigma_M}(\cdot)$ denotes the density of (\ref{doblefunction}).

Let $P_p((k+1)T)$ be the preventive maintenance probability at time $(k+1)T$. That is, 
$$P_p((k+1)T)=P\left(kT \leq V_{[1]} < (k+1)T \leq W_{[1]}\right). $$
This probability is developed as follows. 
\begin{eqnarray*} \nonumber
 P_p((k+1)T) &=& \int_{kT}^{(k+1)T}f_{V_{[1]}}(u)P\left((k+1)T \leq W_{[1]}|V_{[1]}=u\right)~du \\
&=& \int_{kT}^{(k+1)T}f_{V_{[1]}}(u)\bar{F}_{\sigma_L-\sigma_M}((k+1)T-u)P\left(N_{L-M}(u,(k+1)T)=0\right)~du. 
 \end{eqnarray*}
Since $\left\{N_{L-M}(t), \, t \geq 0\right\}$ is a Cox process, we get that
\begin{equation*}
P(N_{L-M}(u,(k+1)T)=0)=\mathbb{E}\left[\exp\left\{-\int_{u}^{(k+1)T}\lambda_M(v)F_{\sigma_L-\sigma_M}((k+1)T-v) ~dv\right\}\right]. 
\end{equation*}
Let $g(u,(K+1)T)$ be the following expectation
\begin{equation*}
g(u,(K+1)T)=\mathbb{E}\left[\exp\left\{-\int_{u}^{(k+1)T}\lambda_M(v)F_{\sigma_L-\sigma_M}((k+1)T-v)~ dv \right\}\right]. 
\end{equation*}
Using Eq. (\ref{lambdaM2})
\begin{equation*}
\lambda_M(t)=\lambda_0F_{\sigma_M}(t)+\sum_{i=1}^{N(t)} \int_{0}^{t} e^{-\delta(u-T_i)}f_{\sigma_M}(t-u) ~du, 
\end{equation*}
and given $N((k+1)T)=n$, for $n>0$, we get that 
 \begin{eqnarray*} \nonumber
&& g_n(u,(K+1)T) = \exp\left\{-\lambda_0 \int_{u}^{(k+1)T} F_{\sigma_M}(v)F_{\sigma_L-\sigma_M}((k+1)T-v) ~dv\right\} \\ \nonumber && \mathbb{E}\left[\exp\left\{-\int_{u}^{(k+1)T} \left(\int_{w=0}^{v}e^{-\delta(w-U)}f_{\sigma_M}(v-w) \mathbf{1}_{\left\{w >U\right\}}\right)F_{\sigma_L-\sigma_M}((k+1)T-v)\right\}\right]^n \\
&& = D_1(u,(K+1)T)D_2(u,(K+1)T), 
\end{eqnarray*} 
 where $U$ is a uniform variable in $(0,(k+1)T)$ and $D_1((K+1)T)$ and $D_2((K+1)T)$ are given by
 \begin{equation} \label{D1}
 D_1(u,(k+1)T) = \exp\left\{-\lambda_0 \int_{u}^{(k+1)T} F_{\sigma_M}(v)F_{\sigma_L-\sigma_M}((k+1)T-v) ~dv\right\}, 
 \end{equation}
 and
 \begin{eqnarray*} \nonumber
D_2(u,(K+1)T) &=& \left(\int_{s=0}^{(k+1)T}\frac{1}{(k+1)T}\exp\left\{-\int_{v=u}^{(k+1)T} \mathcal{L}\sigma_M(v-s)F_{\sigma_L-\sigma_M}((k+1)T-v)\mathbf{1}_{\left\{v>s\right\}} ~dv\right\}\right)^n \\ 
&=& \frac{B(u,(K+1)T)^n}{(k+1)^n T^n}
 \end{eqnarray*}
where
$$B(u,(K+1)T)=\int_{s=0}^{(k+1)T}  \exp\left\{-\int_{u}^{(k+1)T} \mathcal{L}\sigma_M(v-s)F_{\sigma_L-\sigma_M}((k+1)T-v) ~dv\right\} ~ds, $$
and
$$\mathcal{L}\sigma_M(x)=\int_{0}^{x} e^{-\delta z} f_{\sigma_M}(x-z)~dz.  $$ 
Finally, the preventive maintenance probability at time $(k+1)T$ is given by
 \begin{eqnarray} \nonumber
 g(u,(K+1)T) &=& \sum_{n=0}^{\infty} g_n(u,(k+1)T)P[N((k+1)T)=n] \\ \nonumber
 &=& \exp(-\mu (K+1)T) D_1(u,(k+1)T) \\ \nonumber
 &+& \exp(-\mu (K+1)T) D_1(u,(k+1)T) \sum_{n=1}^{\infty} \frac{B(u,(K+1)T)^n}{k+1)^n T^n} \frac{\mu^n (k+1)^n T^n}{n!} \\ \label{gu}
 &=& \exp(-\mu (K+1)T) D_1(u,(k+1)T) \exp(\mu B(u,(k+1)T)), 
 \end{eqnarray} 
 where $D_1(u,(k+1)T)$ is given by (\ref{D1}). 
Finally, the probability of a maintenance preventive at time $(k+1)T$ is given by
 \begin{equation} \label{probabilidadpreventiva}
 P_p((k+1)T)= \int_{kT}^{(k+1)T}f_{V_{[1]}}(u) \bar{F}_{\sigma_L-\sigma_M}((k+1)T-u) g(u,(K+1)T) ~du. 
 \end{equation}
 
 \subsection*{Corrective maintenance probability}
Let $P_c((k+1)T)$ be the probability of a corrective maintenance. This probability is given by
\begin{equation*}
P_c((K+1)T)=P(kT\leq V_{[1]} \leq W_{[1]} \leq (K+1)T). 
\end{equation*} 
Using (\ref{probabilidadpreventiva}), we get that
\begin{eqnarray}\label{probabilidadcorrectiva}
P_c((k+1)T)=\int_{kT}^{(k+1)T}f_{V_{[1]}}(u) \left(\int_{u}^{(K+1)T} \frac{d}{dv} \left(\bar{F}_{\sigma_L-\sigma_M}((k+1)T-v) g(v,(K+1)T) \right)~dv \right)~du, 
\end{eqnarray}
where $g(v,(K+1)T)$ is given by (\ref{gu}). 
\subsection{Expected downtime in a replacement cycle}
Let $E_d(kT,(k+1)T)$ be the expected downtime in $(KT, (k+1)T)$. This expectation is given by
\begin{eqnarray} \label{expecteddowntime}
&& E_d(kT,(k+1)T) = \\ \nonumber
&& \int_{kT}^{(k+1)T}f_{V_{[1]}}(u) \left(\int_{u}^{(K+1)T} \frac{d}{dv} \left(\bar{F}_{\sigma_L-\sigma_M}((k+1)T-v) g(v,(K+1)T) \right)\left((K+1)T-v\right)~dv \right)~du, 
\end{eqnarray}
where $g(v,(K+1)T)$ is given by (\ref{gu})

To evaluate the performance of the maintenance strategy, the
minimization of the asymptotic cost rate is applied. The renewal
theorem (see Tijms \cite{Tijms} for more details) provides a simple
expression of this asymptotic cost rate since it reduces the renewal
process to the first renewal. In this paper, a renewal cycle is defined as the period of time between two consecutive replacements of the system. The expected cost rate for this maintenance model is given by
\begin{equation*}
C(T,M)=\frac{\mathbb{E}[C]}{\mathbb{E}[R]}
\end{equation*}
where $\mathbb{E}[C]$ denotes the expected cost in a replacement cycle and $\mathbb{E}[R]$
denotes the expected time to a replacement cycle. Hence
\begin{equation} \label{cost}
C(T,M) =\frac{C_c \sum_{k=1}^{\infty} P_c(kT)+C_p \sum_{k=1}^{\infty} P_p(kT)+C_I\mathbb{E}[N_I]+\sum_{k=0}^{\infty}C_d\mathbb{E}_d(kT,(k+1)T)}{\mathbb{E}[R]}, 
\end{equation}
where $\mathbb{E}[R]$, $\mathbb{E}[N_I]$, $P_p(kT)$, $P_c(kT)$, and $E_d(kT, (k+1)T)$ are given by Eq. (\ref{expectedlength}), Eq. (\ref{expectedinspection}), Eq. (\ref{probabilidadpreventiva}), Eq. (\ref{probabilidadcorrectiva}) and Eq. (\ref{expecteddowntime}) respectively. 
The search of the optimal maintenance strategy is reduced to
find the values $T_{opt}$ and $M_{opt}$ that minimize the function $C(T,M)$ given by Eq. (\ref{cost}). Next section focuses on numerical examples.

\section{Numerical examples}  \label{f}

{\color{black}
In this section, numerical examples obtained by simulation of the stochastic processes involved in the model are shown. {\color{black} The expression (\ref{cost}) for the objective cost function is difficult to manage it since it evolves infinity sums. A classical way to evaluate the objective cost function is to perform simulations on the points of the mesh and to find the optimal combination by visualization (see \cite{Huynh}) or using optimization metaheuristic methods such as genetic algorithms (\cite{Barcena}). We analyze the optimal maintenance strategy considering a model without heterogeneity and a model with heterogeneity. }
}

\subsection{$\beta$ deterministic}
We assume that degradation processes start according to a stochastic intensity given by
\begin{equation} \label{intensityst}
\lambda(t)=1+\sum_{i=1}^{N(t)}e^{-0.5(t-T_i)}, \quad t \geq 0, 
\end{equation} 
where $N(t)$ denotes the Poisson process {\color{black} of shocks} with parameter $\mu=2$ {\color{black} shocks per time unit}. The degradation processes degrade according to a gamma process with shape parameter $\alpha=1.1$ and scale parameter $\beta=1.4$. It is assumed that the system fails when a degradation process exceeds the failure threshold $L=10$. The following costs are also assumed (m.u. monetary units, t.u. time units )
\begin{equation} \label{seq_cost}
C_p=100 \, m.u., \quad C_c=200 \, m.u., \quad C_I=50 \,  m.u., \quad C_d=60 \,  m.u. \, per \,  t.u. . 
\end{equation}
The search of the optimal maintenance policy corresponds to find the pair $(T,M)$ that optimize the function $C(T,M)$ given by (\ref{cost}). That is, to find $T_{opt}$ and $M_{opt}$ that fulfill
$$C(T_{opt}, M_{opt})=\inf \left\{C(T,M), \ \ M \leq L \right  \}$$


To {\color{black} obtain} Figure \ref{costefig2}, 6,000 simulations were performed on a grid of size $10$ in the interval $(0,25)$ for $T$ and a grid of size $8$ in $(0,10)$ for $M$. By visualization, the optimal values obtained for the time between inspections and the preventive threshold are $T_{opt}=6.3333$ and $M_{opt}=6.1429$ with an optimal expected cost of  $35.3005$ monetary units per unit time.


{\color{black} Several sensitivity analysis were conducted to determine how the stochastic intensity affects the proposed model and to analyze the robustness of the solution when different parameters vary.} Assuming that the degradation processes start according to the intensity given by Eq. (\ref{intensityst}), assuming that the system fails when the degradation level of a process exceeds the failure threshold $L=10$ and with $\beta$ deterministic, Table \ref{sensibilidad3} {\color{black} shows the influence of the shape and scale parameters of the gamma process on the optimal expected cost rate with a shaded grey scale when $\alpha$ and $\beta$ vary from $1$ to $1.9$ increasing by $0.15$ at each step. Along with Table \ref{sensibilidad3}, Tables \ref{sensibilidad4} and \ref{sensibilidad5} show the optimal values of $T$ and $M$ for each combination of values of $\alpha$ and $\beta$. Table \ref{sensibilidad_costes}  studies the influence of preventive and corrective maintenance costs on the total optimal cost. To observe the variation of the cost for preventive maintenance, 9 points have been considered in the interval $[90,110]$, while the cost for corrective maintenance has been inspected in the interval $[190,210]$, also with 9 points. This table has been obtained by fixing the optimal values for $T$ and $M$ obtained previously.} 



{\color{black}
Table \ref{sensibilidad3} shows that $\alpha$ and $\beta$ have both effect on the optimal expected cost. It is worth noting that more differences in the optimal expected cost are found when the scale parameter varies.
Similarly, as we can see in Table \ref{sensibilidad4}, both parameters have effect on the optimal value for the time between inspections $T$. However Table \ref{sensibilidad5} does not show a clear influence of the parameters in the case of the optimal value for the preventive threshold $M$. The same can be said for Table \ref{sensibilidad_costes} regarding to the influence of maintenance costs in the model.

}

{\color{black}
As it is shown in Table \ref{sensibilidad3}, the optimal expected cost decreases as $\alpha$ and $\beta$ decreases. However, as it is shown in Table \ref{sensibilidad4}, the optimal time between inspections decreases when $\alpha$ and $\beta$ increases. Higher values for the scale and shape parameter of the gamma process mean a faster deterioration hence inspections tend to be more frequent to reduce costs. 

}

\begin{figure}[tbph]
\begin{center}
\includegraphics[width=0.7\textwidth]{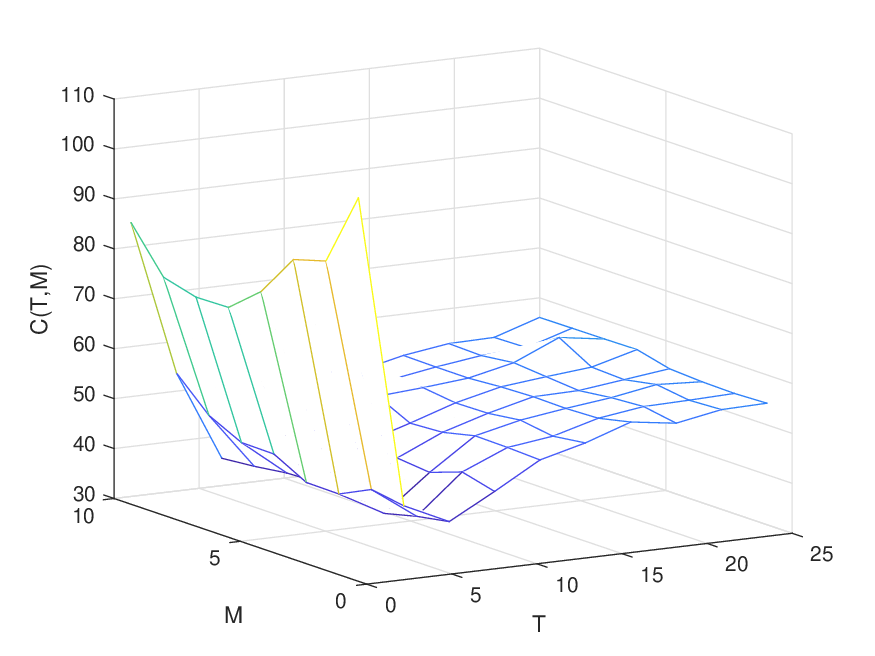}
\caption{{\protect\footnotesize {Expected cost rate versus $T$ and $M$.}}}
 \label{costefig2}
\end{center}
\end{figure}

\begin{table}[tbph]
\begin{center}
\begin{tabular}{|c|c|c|c|c|c|c|c|}
\hline
\backslashbox{$\alpha$}{$\beta$} & 1 & 1.15 & 1.3 & 1.45 & 1.6 & 1.75 & 1.9  \\
\hline
1  & \cellcolor{gris0} 19.8048 & \cellcolor{gris0}  21.4489 & \cellcolor{gris1} 25.7375  & \cellcolor{gris1} 27.8152 &\cellcolor{gris1} 30.7940 & \cellcolor{gris1}  30.7778 & \cellcolor{gris2} 35.7045 \\
1.15  & \cellcolor{gris0} 22.4085  & \cellcolor{gris0} 24.1170 & \cellcolor{gris1} 28.1522  & \cellcolor{gris2} 31.0347 & \cellcolor{gris2} 32.3892   & \cellcolor{gris2}  35.7728  & \cellcolor{gris3}  39.1335 \\
1.3  & \cellcolor{gris0} 24.9069 & \cellcolor{gris1} 27.3980 &  \cellcolor{gris2} 31.9920 & \cellcolor{gris2} 33.3075 &  \cellcolor{gris2} 36.1599  & \cellcolor{gris3}  39.5103  & \cellcolor{gris3}  41.4842 \\
1.45  & \cellcolor{gris1} 27.6355 & \cellcolor{gris1}  29.7365 & \cellcolor{gris2} 34.2041 & \cellcolor{gris2} 34.8008  & \cellcolor{gris3} 37.3992  &  \cellcolor{gris3} 40.4628  & \cellcolor{gris4} 46.5984 \\
1.6  & \cellcolor{gris1}  29.0795 & \cellcolor{gris1} 30.4875 &  \cellcolor{gris2} 34.9262  & \cellcolor{gris3} 38.3855  & \cellcolor{gris3}  41.6051 & \cellcolor{gris4}  46.3327  & \cellcolor{gris4} 46.8082 \\
1.75 & \cellcolor{gris1} 30.4135 & \cellcolor{gris2} 35.9786 & \cellcolor{gris3} 37.9817  & \cellcolor{gris4} 44.6625 & \cellcolor{gris4} 45.7936  & \cellcolor{gris4} 48.6405  & \cellcolor{gris5} 52.9519 \\
1.9 & \cellcolor{gris1}  30.5336 & \cellcolor{gris2} 36.1878 & \cellcolor{gris3} 42.7939 & \cellcolor{gris4} 45.5782  &  \cellcolor{gris5} 49.2110  & \cellcolor{gris5}  52.4083 & \cellcolor{gris6} 60.1414 \\
\hline
\end{tabular}
\caption{Sensitivity analysis on the values of  $\alpha$ and $\beta$ for the optimal expected cost rate.} \label{sensibilidad3}
\end{center}
\end{table}

\begin{table}[tbph]
\begin{center}
\begin{tabular}{|c|c|c|c|c|c|c|c|}
\hline
\backslashbox{$\alpha$}{$\beta$} & 1 & 1.15 & 1.3 & 1.45 & 1.6 & 1.75 & 1.9  \\
\hline
1  & \cellcolor{gris6} 8.2220 & \cellcolor{gris6}  8.0979 & \cellcolor{gris4} 6.7767  & \cellcolor{gris4} 6.2052 &\cellcolor{gris3} 6.0996 & \cellcolor{gris4}  6.4665 & \cellcolor{gris1} 4.5638 \\
1.15  & \cellcolor{gris6} 7.5315  & \cellcolor{gris5} 7.1678 & \cellcolor{gris4} 6.2315  & \cellcolor{gris3} 5.7359 & \cellcolor{gris2} 5.1743   & \cellcolor{gris2} 5.2126  & \cellcolor{gris1} 4.6818 \\
1.3  & \cellcolor{gris6} 7.4103 & \cellcolor{gris3} 6.0552 &  \cellcolor{gris4} 6.5784 & \cellcolor{gris2} 5.3184 &  \cellcolor{gris1} 5.0614  & \cellcolor{gris2}  5.1942  & \cellcolor{gris1}  4.8545 \\
1.45  & \cellcolor{gris4} 6.3486 & \cellcolor{gris3}  5.9120 & \cellcolor{gris2} 5.4764 & \cellcolor{gris2} 5.2139  & \cellcolor{gris1} 4.6110  &  \cellcolor{gris1} 4.5256  & \cellcolor{gris1} 4.8004 \\
1.6  & \cellcolor{gris3}  5.8572 & \cellcolor{gris3} 5.7846 &  \cellcolor{gris2} 5.2789  & \cellcolor{gris0} 4.4831  & \cellcolor{gris0}  4.1643 & \cellcolor{gris0}  4.3130  & \cellcolor{gris0} 3.9727 \\
1.75 & \cellcolor{gris2} 5.2285 & \cellcolor{gris1} 4.9454 & \cellcolor{gris1} 4.6661  & \cellcolor{gris0} 4.3905 & \cellcolor{gris0} 4.2302  & \cellcolor{gris0} 3.9181  & \cellcolor{gris0} 4.2534 \\
1.9 & \cellcolor{gris3}  5.6905 & \cellcolor{gris1} 4.6343 & \cellcolor{gris0} 4.3063 & \cellcolor{gris0} 3.9870  &  \cellcolor{gris0} 3.9727  & \cellcolor{gris0}  4.2653 & \cellcolor{gris1} 4.7039 \\
\hline
\end{tabular}
\caption{Sensitivity analysis on the values $\alpha$ and $\beta$ for the optimal value of $T$.} \label{sensibilidad4}
\end{center}
\end{table}

\begin{table}[tbph]
\begin{center}
\begin{tabular}{|c|c|c|c|c|c|c|c|}
\hline
\backslashbox{$\alpha$}{$\beta$} & 1 & 1.15 & 1.3 & 1.45 & 1.6 & 1.75 & 1.9  \\
\hline
1  & \cellcolor{gris6} 3.9426 & \cellcolor{gris2}  2.1435 & \cellcolor{gris3} 2.6667  & \cellcolor{gris1} 1.6410 &\cellcolor{gris1} 1.7250 & \cellcolor{gris3}  2.7337 & \cellcolor{gris1} 1.6285 \\
1.15  & \cellcolor{gris0} 1.2083  & \cellcolor{gris3} 2.6261 & \cellcolor{gris2} 2.1233  & \cellcolor{gris3} 2.6294 & \cellcolor{gris3} 2.6088   & \cellcolor{gris1}  1.6989  & \cellcolor{gris4}  2.9067 \\
1.3  & \cellcolor{gris6} 4.0309 & \cellcolor{gris1} 1.3344 &  \cellcolor{gris4} 2.9083 & \cellcolor{gris3} 2.5836 &  \cellcolor{gris2} 2.2619  & \cellcolor{gris2}  2.0369  & \cellcolor{gris3}  2.3512 \\
1.45  & \cellcolor{gris0} 1.1294 & \cellcolor{gris1}  1.6428 & \cellcolor{gris0} 1.3248 & \cellcolor{gris3} 2.4431  & \cellcolor{gris1} 1.8680  &  \cellcolor{gris1} 1.5133  & \cellcolor{gris1} 1.8599 \\
1.6  & \cellcolor{gris2}  2.2714 & \cellcolor{gris1} 1.8112 &  \cellcolor{gris3} 2.3257  & \cellcolor{gris0} 1.3541  & \cellcolor{gris4}  2.8780 & \cellcolor{gris0}  1.4741  & \cellcolor{gris0} 1.3602 \\
1.75 & \cellcolor{gris2} 1.9457 & \cellcolor{gris2} 2.2737 & \cellcolor{gris0} 1.1288  & \cellcolor{gris0} 1.2540 & \cellcolor{gris2} 2.0906  & \cellcolor{gris0} 1.1675  & \cellcolor{gris1} 1.9039 \\
1.9 & \cellcolor{gris3}  2.3188 & \cellcolor{gris4} 2.8609 & \cellcolor{gris4} 3.0805 & \cellcolor{gris0} 1.2140  &  \cellcolor{gris2} 2.1427 & \cellcolor{gris1}  1.6137 & \cellcolor{gris1} 1.5923 \\
\hline
\end{tabular}
\caption{Sensitivity analysis on the values of $\alpha$ and $\beta$ for the optimal value of $M$.} \label{sensibilidad5}
\end{center}
\end{table}

\begin{table}[tbph]
\begin{center}
\begin{tabular}{|c|c|c|c|c|c|c|c|c|c|}
\hline
\backslashbox{$C_c$}{$C_p$} & 90 & 92.5 &  95 & 97.5 & 100 & 102.5 & 105 & 107.5 & 110 \\
\hline
190  & \cellcolor{gris0} 32.05 & \cellcolor{gris0}  32.11 & \cellcolor{gris1} 35.77  & \cellcolor{gris1} 34.97 &\cellcolor{gris1} 36.74 & \cellcolor{gris1}  35.35 & \cellcolor{gris1} 36.91 & \cellcolor{gris1} 36.56 & \cellcolor{gris1}  35.76 \\
192.5  & \cellcolor{gris0} 33.32 & \cellcolor{gris2} 38.32 & \cellcolor{gris0} 33.63  & \cellcolor{gris0} 31.32 & \cellcolor{gris1} 35.63   & \cellcolor{gris3}  40.54  & \cellcolor{gris2}  37.32 &  \cellcolor{gris3} 41.83 & \cellcolor{gris2}  38.93 \\
195  & \cellcolor{gris5} 46.17 & \cellcolor{gris0} 33.80 &  \cellcolor{gris6} 50.28 & \cellcolor{gris6} 51.64 &  \cellcolor{gris6} 51.88  & \cellcolor{gris6}  52.29  & \cellcolor{gris1}  36.58 & \cellcolor{gris2}  37.80 & \cellcolor{gris3}  40.59 \\
197.5  & \cellcolor{gris2} 38.60 & \cellcolor{gris3}  41.46 & \cellcolor{gris2} 37.50 & \cellcolor{gris2} 37.91  & \cellcolor{gris1} 36.74  &  \cellcolor{gris2} 38.06  & \cellcolor{gris1} 36.38 & \cellcolor{gris4} 44.02 & \cellcolor{gris2}  38.95 \\
200 & \cellcolor{gris0}  33.52 & \cellcolor{gris1} 36.48 &  \cellcolor{gris3} 40.35  & \cellcolor{gris2} 39.34  & \cellcolor{gris1}  35.00 & \cellcolor{gris2}  37.06  & \cellcolor{gris0} 33.76 & \cellcolor{gris0} 33.77 &  \cellcolor{gris3}  40.38 \\
202.5 & \cellcolor{gris0} 34.21 & \cellcolor{gris1} 35.43 & \cellcolor{gris2} 39.97  & \cellcolor{gris2} 38.45 & \cellcolor{gris2} 39.47  & \cellcolor{gris1} 35.05  & \cellcolor{gris1} 36.99 & \cellcolor{gris1} 36.51  &  \cellcolor{gris1} 36.97 \\
205 & \cellcolor{gris1}  34.38 & \cellcolor{gris1} 35.60 & \cellcolor{gris1} 36.67 & \cellcolor{gris1} 35.68  &  \cellcolor{gris1} 35.27 & \cellcolor{gris2}  37.16 & \cellcolor{gris0} 33.11 &  \cellcolor{gris4} 45.96 & \cellcolor{gris6} 49.05 \\
207.5 & \cellcolor{gris1}  35.04 & \cellcolor{gris1} 34.76 & \cellcolor{gris0} 32.51 & \cellcolor{gris1} 35.66  &  \cellcolor{gris3} 40.10 & \cellcolor{gris1}  36.01 & \cellcolor{gris3} 40.27 &  \cellcolor{gris2} 37.81 &  \cellcolor{gris2} 38.92 \\
210 & \cellcolor{gris1}  36.57 & \cellcolor{gris2} 37.45 & \cellcolor{gris2} 37.00 & \cellcolor{gris2} 38.76  &  \cellcolor{gris2} 37.74 & \cellcolor{gris2}  39.62 & \cellcolor{gris4} 43.93 & \cellcolor{gris2} 39.07  & \cellcolor{gris2} 39.50 \\
\hline
\end{tabular}
\caption{Sensitivity analysis on the values of $C_c$ and $C_p$ for the optimal expected cost rate.} \label{sensibilidad_costes}
\end{center}
\end{table}

\subsection{Random effects model}

This section assumes that the scale parameter $\beta$ is random.  As before, the degradation processes start according to the stochastic failure rate given by Eq. (\ref{intensityst}) being $N(t)$ a Poisson process with parameter $\mu=2$. We assume that the shape parameter of the gamma process is equal to $\alpha=1.1$ and the scale parameter $\beta^{-1}$ follows a uniform distribution in the interval $(1/1.4-\alpha^*,1/1.4+\alpha^*)$, with $\alpha^*=0.1$.  As above, the system fails when the level of  degradation of a process exceeds $L=10$ and the sequence of costs given by Eq. (\ref{seq_cost}) is imposed. To analyse the optimal maintenance strategy, a grid of size 10 in the interval $(0,25)$ has been considered for $T$. Similarly, a grid of size 8 has been considered for $M$ in $(0,10)$. For each combination of $M$ and $T$, 6,000 simulations were performed. Figure \ref{costefig} shows the results of the expected cost rate versus $T$ and $M$. By inspection, the optimal values are obtained for $T_{opt}=6.3333$ and $M_{opt}=4.8571$ with an optimal expected cost rate equals to $37.1962$ monetary units per {\color{black} time unit}. {\color{black} Compared to the model without heterogeneity of the previous subsection,  both objective cost functions show the same pattern. The case without heterogeneity shows a lower optimal expected cost rate ($C(T_{opt}, M_{opt})=35.3005)$ monetary units per time unit than the case with heterogeneity $C(T_{opt}, M_{opt})=37.1962)$ monetary units per time unit (both processes have the same expectation).    }


{\color{black}

The sensitivity analysis of the expected cost rate and the optimal values for $T$ and $M$ is summarized in Tables \ref{sensibilidad6}, \ref{sensibilidad7} and \ref{sensibilidad8}, 
with $\alpha$ varying from $1$ to $1.9$ with increments of $0.15$ units and $\beta^{-1}$ following a uniform distribution in $(b-\alpha^*,b+\alpha^*)$, being $\alpha^*=0.1$ and $b$ varying in the same manner as $\alpha$, from $1$ to $1.9$ also with increments of $0.15$ units. Table \ref{sensibilidad_costes_random} represents the variation in the expected cost rate when the costs for preventive maintenance and corrective maintenance vary, using the values of $T$ and $M$ obtained in the optimization of the model.

As in the model without heterogeneity, the optimal expected cost rate shown in Table \ref{sensibilidad6} also increases as $\alpha$ increases. That is, more deterioration implies higher costs. In this particular example, the random choice of the parameter $\beta$ has the same effect as $\alpha$ on the resulting expected cost rate: it increases when $b$ increases. Similarly, Table \ref{sensibilidad7} shows that the optimal value of $T$ decreases as $\alpha$ increases. An increase in the deterioration implies more frequent inspections. Finally, no trend seems to be found for the optimal values of the preventive threshold $M$ or in the variation of maintenance costs $C_c$ and $C_p$ shown in Tables \ref{sensibilidad8} and \ref{sensibilidad_costes_random}, respectively.

}




%
%


\begin{figure}[tbph]
\begin{center}
\includegraphics[width=0.7\textwidth]{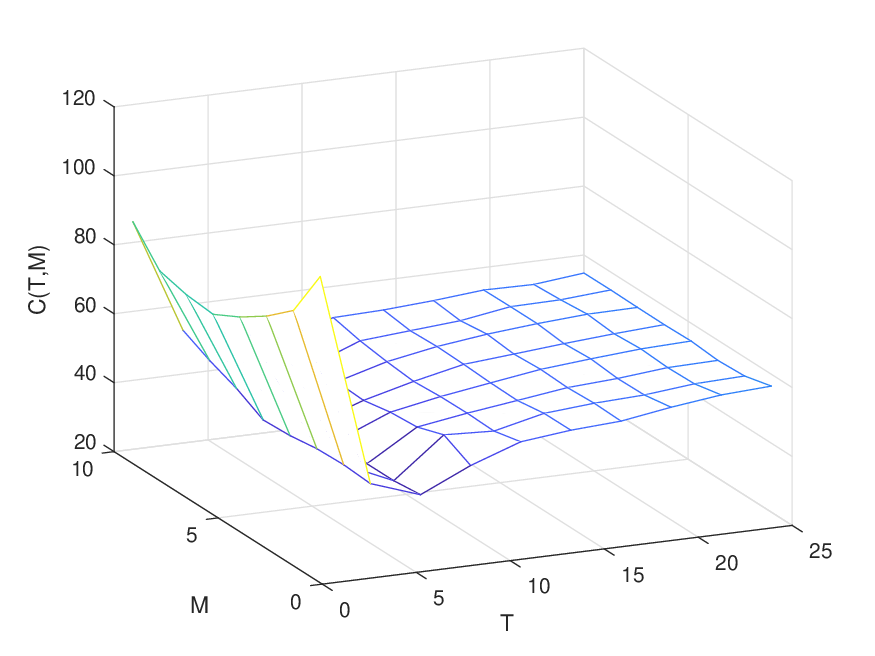}
\caption{{\protect\footnotesize {Expected cost rate versus $T$ and $M$ in a model with random effects.}}}
 \label{costefig}
\end{center}
\end{figure}

\begin{table}[tbph]
\begin{center}
\begin{tabular}{|c|c|c|c|c|c|c|c|c|c|}
\hline
\backslashbox{$\alpha$}{{\color{black} $b$}} & 1 & 1.15 & 1.3 & 1.45 & 1.6 & 1.75 & 1.9   \\
\hline
1  & \cellcolor{gris0} 26.1538 & \cellcolor{gris0}  28.3350 & \cellcolor{gris1} 31.8173  & \cellcolor{gris1} 32.9500 &\cellcolor{gris2} 38.0842 & \cellcolor{gris3}  44.5595 & \cellcolor{gris3} 48.0884  \\
1.15  & \cellcolor{gris0} 27.2921 & \cellcolor{gris0} 30.9587 & \cellcolor{gris1} 34.4670  & \cellcolor{gris2} 38.8324 & \cellcolor{gris3} 45.0903   & \cellcolor{gris3}  46.3951  & \cellcolor{gris4}  54.5271   \\
1.3  & \cellcolor{gris1} 37.6802 & \cellcolor{gris1} 37.1068 &  \cellcolor{gris2} 42.1956 & \cellcolor{gris3} 44.4267 &  \cellcolor{gris4} 49.5072 & \cellcolor{gris5}  56.1529  & \cellcolor{gris5}  58.2646  \\
1.45  & \cellcolor{gris1} 32.5948 & \cellcolor{gris2}  39.7887 & \cellcolor{gris3} 44.3998 & \cellcolor{gris4} 49.9204  & \cellcolor{gris4} 53.8380  &  \cellcolor{gris5} 56.2251  & \cellcolor{gris6} 61.2970  \\
1.6  & \cellcolor{gris1}  37.4109 & \cellcolor{gris2} 42.5415 &  \cellcolor{gris3} 48.5337  & \cellcolor{gris4} 53.8215  & \cellcolor{gris5}  55.9734 & \cellcolor{gris5} 60.1671  & \cellcolor{gris6} 67.9211   \\
1.75 & \cellcolor{gris2} 39.4109 & \cellcolor{gris3} 46.8386 & \cellcolor{gris5} 56.7378  & \cellcolor{gris5} 58.7266 & \cellcolor{gris5} 60.9805  & \cellcolor{gris6} 61.7760  & \cellcolor{gris6} 64.8222    \\
1.9 & \cellcolor{gris3} 46.9104 & \cellcolor{gris4} 51.9782 & \cellcolor{gris5} 57.3185 & \cellcolor{gris5} 60.1173 &  \cellcolor{gris6} 63.9612  & \cellcolor{gris6} 66.7328  & \cellcolor{gris6} 66.8821  \\
\hline
\end{tabular}
\caption{Sensitivity analysis on the values of  $\alpha$ and $b$ for the optimal expected cost rate in a model with random effects.} \label{sensibilidad6}
\end{center}
\end{table}


\begin{table}[tbph]
\begin{center}
\begin{tabular}{|c|c|c|c|c|c|c|c|}
\hline
\backslashbox{$\alpha$}{{\color{black} $b$}} & 1 & 1.15 & 1.3 & 1.45 & 1.6 & 1.75 & 1.9  \\
\hline
1  & \cellcolor{gris2} 8.3615 & \cellcolor{gris2}  9.3511 & \cellcolor{gris3} 10.8538 & \cellcolor{gris3} 11.6696 &\cellcolor{gris5} 13.6147 & \cellcolor{gris6}  13.8070 & \cellcolor{gris6} 14.9346 \\
1.15  & \cellcolor{gris1} 7.4953 & \cellcolor{gris2} 8.3954 &  \cellcolor{gris2} 9.3891 & \cellcolor{gris3} 10.4470 &  \cellcolor{gris4} 11.4916 & \cellcolor{gris5} 12.3747  & \cellcolor{gris6}  13.6643 \\
1.3  & \cellcolor{gris0} 6.7435 & \cellcolor{gris1}  8.0170 & \cellcolor{gris2} 8.8949 & \cellcolor{gris2} 9.2960  & \cellcolor{gris3} 10.3069  &  \cellcolor{gris4} 11.4964  & \cellcolor{gris4} 12.1105 \\
1.45  & \cellcolor{gris0}  6.5243 & \cellcolor{gris1} 6.9606 &  \cellcolor{gris1} 7.9728 & \cellcolor{gris2} 9.0949 & \cellcolor{gris2}  9.4303 & \cellcolor{gris3}  10.1668  & \cellcolor{gris4} 11.6343 \\
1.6  & \cellcolor{gris0} 5.8121  & \cellcolor{gris0} 6.7253 & \cellcolor{gris1} 8.0823  & \cellcolor{gris2} 8.6948 & \cellcolor{gris2} 9.0396   & \cellcolor{gris2} 9.4426  & \cellcolor{gris3} 10.3791 \\
1.75 & \cellcolor{gris0} 5.7501 & \cellcolor{gris0} 6.0871 & \cellcolor{gris1} 6.9204  & \cellcolor{gris1} 7.4151 & \cellcolor{gris1} 8.1554  & \cellcolor{gris2} 8.7310 & \cellcolor{gris3} 9.5734 \\
1.9 & \cellcolor{gris0}  5.4413 & \cellcolor{gris0} 5.8472 & \cellcolor{gris0} 6.4068 & \cellcolor{gris1} 6.8895  & \cellcolor{gris1} 8.0662 & \cellcolor{gris1}  8.1205 & \cellcolor{gris2} 8.8067 \\
\hline
\end{tabular}
\caption{Sensitivity analysis on the values of $\alpha$ and $b$ for the optimal value of $T$ in a model with random effects.} \label{sensibilidad7}
\end{center}
\end{table}

\begin{table}[tbph]
\begin{center}
\begin{tabular}{|c|c|c|c|c|c|c|c|}
\hline
\backslashbox{$\alpha$}{{\color{black} $b$}} & 1 & 1.15 & 1.3 & 1.45 & 1.6 & 1.75 & 1.9  \\
\hline
1  & \cellcolor{gris2} 2.7979 & \cellcolor{gris0}  1.1639 & \cellcolor{gris2} 2.2698  & \cellcolor{gris4} 3.6165 &\cellcolor{gris0} 1.6126 & \cellcolor{gris4}  3.4845 & \cellcolor{gris3} 3.3828 \\
1.15  & \cellcolor{gris0} 1.3415 & \cellcolor{gris6} 5.5418 & \cellcolor{gris1} 1.6695  & \cellcolor{gris3} 3.1643 & \cellcolor{gris1} 1.8576   & \cellcolor{gris2} 2.7601 & \cellcolor{gris4}  3.7015 \\
1.3  & \cellcolor{gris0} 1.4256 & \cellcolor{gris0} 1.1350 &  \cellcolor{gris1} 1.6613 & \cellcolor{gris3} 3.2311 &  \cellcolor{gris3} 3.1447  & \cellcolor{gris1}  1.9884 & \cellcolor{gris2}  2.3213 \\
1.45  & \cellcolor{gris3} 2.9804 & \cellcolor{gris1}  1.7815 & \cellcolor{gris0} 1.5055 & \cellcolor{gris0} 1.4391  & \cellcolor{gris2} 2.4255  &  \cellcolor{gris0} 1.5748  & \cellcolor{gris6} 5.2194 \\
1.6  & \cellcolor{gris1}  1.7922 & \cellcolor{gris2} 2.3719 &  \cellcolor{gris4} 3.4496 & \cellcolor{gris2} 2.5720  & \cellcolor{gris3}  3.2729 & \cellcolor{gris4} 3.6388  & \cellcolor{gris1} 1.6719 \\
1.75 & \cellcolor{gris4} 3.5678 & \cellcolor{gris2} 2.5439 & \cellcolor{gris4} 3.6870 & \cellcolor{gris0} 1.2555 & \cellcolor{gris3} 3.2687  & \cellcolor{gris1} 2.2337  & \cellcolor{gris3} 3.3523 \\
1.9 & \cellcolor{gris0}  1.4376 & \cellcolor{gris3} 3.1995 & \cellcolor{gris0} 1.1186 & \cellcolor{gris1} 2.0668  &  \cellcolor{gris0} 1.3907 & \cellcolor{gris5}  4.2251 & \cellcolor{gris2} 2.5016 \\
\hline
\end{tabular}
\caption{Sensitivity analysis on the values of  $\alpha$ and $b$ for the optimal value of $M$ in a model with random effects.} \label{sensibilidad8}
\end{center}
\end{table}

\begin{table}[tbph]
\begin{center}
\begin{tabular}{|c|c|c|c|c|c|c|c|c|c|}
\hline
\backslashbox{$C_c$}{$C_p$} & 90 & 92.5 &  95 & 97.5 & 100 & 102.5 & 105 & 107.5 & 110 \\
\hline
190  & \cellcolor{gris0} 31.15 & \cellcolor{gris2}  36.48 & \cellcolor{gris1} 33.77 & \cellcolor{gris2} 36.77 &\cellcolor{gris1} 33.09 & \cellcolor{gris1}  33.18 & \cellcolor{gris4} 43.31 & \cellcolor{gris3} 39.84 & \cellcolor{gris1}  33.40 \\
192.5  & \cellcolor{gris2} 37.24 & \cellcolor{gris1} 35.31 & \cellcolor{gris0} 32.65 & \cellcolor{gris1} 33.58 & \cellcolor{gris2} 37.47   & \cellcolor{gris1}  33.56  & \cellcolor{gris1} 34.54 & \cellcolor{gris0}  32.91 & \cellcolor{gris0}  31.88 \\
195  & \cellcolor{gris0} 32.96 & \cellcolor{gris1} 33.19 &  \cellcolor{gris1} 34.71 & \cellcolor{gris0} 31.13 &  \cellcolor{gris0} 32.93  & \cellcolor{gris0}  31.62  & \cellcolor{gris1}  34.08 & \cellcolor{gris1} 34.78 &  \cellcolor{gris1} 33.02 \\
197.5  & \cellcolor{gris1} 33.65 & \cellcolor{gris4}  45.19 & \cellcolor{gris1} 34.26 & \cellcolor{gris2} 36.21  & \cellcolor{gris1} 34.43  &  \cellcolor{gris4} 43.83  & \cellcolor{gris2} 36.55 & \cellcolor{gris4} 42.48 & \cellcolor{gris1}  34.34 \\
200 & \cellcolor{gris1}  33.51 & \cellcolor{gris4} 43.01 &  \cellcolor{gris3} 39.67 & \cellcolor{gris0} 31.93 & \cellcolor{gris3}  39.48 & \cellcolor{gris2}  36.13  & \cellcolor{gris1} 34.60 & \cellcolor{gris1} 33.81 & \cellcolor{gris2}  37.08 \\
202.5 & \cellcolor{gris1} 34.02 & \cellcolor{gris2} 36.24 & \cellcolor{gris1} 33.92  & \cellcolor{gris1} 35.40 & \cellcolor{gris2} 36.48  & \cellcolor{gris1} 33.11  & \cellcolor{gris1} 35.94 & \cellcolor{gris1} 34.56  &  \cellcolor{gris1}  34.61 \\
205 & \cellcolor{gris1}  35.99 & \cellcolor{gris1} 33.67 & \cellcolor{gris4} 44.57 & \cellcolor{gris2} 38.59  &  \cellcolor{gris1} 35.89 & \cellcolor{gris3}  39.22 & \cellcolor{gris1} 33.87 &  \cellcolor{gris3} 41.78 & \cellcolor{gris2}  38.03 \\
207.5 & \cellcolor{gris1}  34.59 & \cellcolor{gris0} 30.98 & \cellcolor{gris1} 33.59 & \cellcolor{gris1} 35.03  &  \cellcolor{gris2} 36.18 & \cellcolor{gris4}  43.77 & \cellcolor{gris6} 50.11 &  \cellcolor{gris2} 36.11 & \cellcolor{gris1}  33.27 \\
210 & \cellcolor{gris3}  39.02 & \cellcolor{gris1} 34.31 & \cellcolor{gris2} 37.23 & \cellcolor{gris1} 34.18  &  \cellcolor{gris1} 34.70 & \cellcolor{gris1}  35.04 & \cellcolor{gris2} 36.52 & \cellcolor{gris3} 39.06  & \cellcolor{gris3}  39.01 \\
\hline
\end{tabular}
\caption{Sensitivity analysis on the values of $C_c$ and $C_p$ for the optimal expected cost rate in a model with random effects.} \label{sensibilidad_costes_random}
\end{center}
\end{table}

\section{Conclusions, future extensions and limitations of this work}

A {\color{black} system subject to multiple degradation processes is analyzed in this paper. Degradation processes start according to a shot noise Cox process and grow according to homogeneous gamma processes Under this framework, the combined model of initiation and growth is modelled as a Cox process. Using properties of the Cox processes, the distribution to the system lifetime is obtained.} It is shown that, in absence of maintenance, the system lifetime {\color{black} distribution} is increasing failure rate hence a preventive maintenance policy is worth implementing to
improve the system reliability. The analysis is also completed with degradation process-specific heterogeneity considering a random effects model {\color{black} using a uniform distribution.} A classical maintenance strategy is implemented for this system. As before, the probability expressions of the objective cost function are closed forms that allow intuitive interpretations. It is due to the mathematical tractability of the shot noise Cox process induced by a Poisson process. 

Although in this paper starting points of the degradation processes follow a shot noise Cox process, the result can be extended considering a different Cox process such as a Weibull renewal process. Or, even, a non-Cox distribution such as a Hawkes process. Another crucial assumption in this paper is that the degradation processes evolve independently and according to the same degradation pattern. For future research, different degradation patters for the degradation processes and dependence between the processes could be considered.

{\color{black} With respect to the limitations of the paper, the most important limitation refers to the lack of real data that support the proposed maintenance model.}

\section*{Acknowledgements} This research was supported by Ministerio de Ciencia e Innovación, Spain (Project PGC2018-094964-B-I00).

\end{document}